\documentclass[11pt, reqno]{amsart}

\usepackage{amssymb}    
\usepackage{amsmath}    
\usepackage{amsthm}     
\usepackage{amssymb}
\usepackage{mathrsfs}

\usepackage[margin=1in]{geometry}
\numberwithin{equation}{section}

\newtheorem{theorem}{Theorem}[section]
\newtheorem{lemma}[theorem]{Lemma}
\newtheorem{proposition}[theorem]{Proposition}
\newtheorem{corollary}[theorem]{Corollary}

\newenvironment{prf}[1]{\trivlist
\item[\hskip
\labelsep{\it #1.\hspace*{.3em}}]}{
\endtrivlist}

\newtheorem{predefinition}[theorem]{Definition}

\newtheorem{preremark}[theorem]{Remark}
\newenvironment{remark}{\begin{preremark}\rm}{\end{preremark}}
\newtheorem{prenotation}[theorem]{Notation}
\newenvironment{notation}{\begin{prenotation}\rm}{\end{prenotation}}
\newtheorem{preexample}[theorem]{Example}
\newenvironment{example}{\begin{preexample}\rm}{\end{preexample}}
\newtheorem{preclaim}[theorem]{Claim}

\newtheorem{prequestion}[theorem]{Question}

 \makeatletter
\def\emppsubsection{\@startsection{subsection}{2}{\z@}{-3.25ex plus -1ex minus -.2ex}{-1em}{\bf}}

\makeatother

\newcommand \nil {{\text{\rm nil}}}
\newcommand \sss{{\text{\rm ss}}}
\newcommand \ZZ {{\mathbb Z}}
\newcommand  \FF {{\mathbb F}}
\newcommand \GG {{\mathbb G}}
\newcommand \dime {\mathop{\rm dim}}
\newcommand \di {{\textstyle\mathop{\rm div}}}
\newcommand \Hom {\mathop{\rm Hom}}

\DeclareMathOperator{\HH}{H}

\def\Cal{\mathcal}
\def\Bb{\mathbb}

\def\bE{{\Bb E}}
\def\bP{{\Bb P}}
\def\Z{{\Bb Z}}
\def\NN{{\Bb N}}
\def\cO{{\Cal O}}
\def\cU{{\Cal U}}
\def\cF{{\Cal F}}

\def\cov{{\pi}}		

\def\dx{{\,dx}}
\def\dy{{\,dy}}
\def\xalpha{{(x-\alpha)^{-1}}}

\def\set#1{\left\{#1\right\}}
\def\gen#1{\bigl\langle#1\bigr\rangle}

\DeclareMathOperator{\CH}{{H}}
\DeclareMathOperator{\CB}{{B}}
\DeclareMathOperator{\CZ}{{Z}}
\DeclareMathOperator{\CC}{{C}}
\DeclareMathOperator{\OOp}{{\Theta}}

\def\HdR{{\CH_\text{\rm dR}^1}}

\def\BdR{{\CB_\text{\rm dR}^1}}
\def\ZdR{{\CZ_\text{\rm dR}^1}}

\def\cC{{\mathscr C}}

\begin{document}

\title{Ekedahl-Oort strata of hyperelliptic curves in characteristic $2$}
\author{Arsen Elkin}
\address{Mathematics Institute\\
University of Warwick\\
Coventry CV4 7AL\\
UK}
\email{A.Elkin@warwick.ac.uk}

\author{Rachel Pries}
\address{Department of Mathematics\\
Colorado State University\\
Fort Collins, CO 80523\\
USA}
\email{pries@math.colostate.edu}

\thanks{The first author was partially supported by the Marie Curie Incoming International Fellowship PIIF-GA-2009-236606.  The second author was partially supported by NSF grant DMS-07-01303. The authors would like to thank Jeff Achter for helpful comments.}

\date{}

\begin{abstract}
Suppose $X$ is a hyperelliptic curve of genus $g$ defined over an algebraically closed field $k$ of characteristic $p=2$.  
We prove that the de Rham cohomology of $X$ decomposes into pieces 
indexed by the branch points of the hyperelliptic cover. 
This allows us to compute the isomorphism class of the $2$-torsion group scheme $J_X[2]$ of the Jacobian of $X$
in terms of the Ekedahl-Oort type.
The interesting feature is that $J_X[2]$ depends only on some discrete invariants of $X$, 
namely, on the ramification invariants associated with the branch points.
We give a complete classification of the group schemes which occur 
as the $2$-torsion group schemes of Jacobians of hyperelliptic $k$-curves of arbitrary genus.
\end{abstract}

\maketitle

\section{Introduction}

Suppose $k$ is an algebraically closed field of characteristic $p>0$.
There are several important stratifications of the moduli space ${\mathcal A}_g$ of
principally polarized abelian varieties of dimension $g$ defined over $k$, including the Ekedahl-Oort stratification.
The Ekedahl-Oort type characterizes the $p$-torsion group scheme of the corresponding abelian varieties, 
and, in particular, determines invariants of the group scheme such as the $p$-rank and $a$-number.
It is defined by the interaction between 
the Frobenius $F$ and Verschiebung $V$ operators on the $p$-torsion group scheme.
Very little is known about how the Ekedahl-Oort strata intersect the Torelli locus of Jacobians of curves.
In particular, one would like to know which group schemes occur as the $p$-torsion $J_X[p]$ 
of the Jacobian $J_X$ of a curve $X$ of genus $g$.

In this paper, we completely answer this question for hyperelliptic $k$-curves $X$ of arbitrary genus when $k$ has characteristic $p=2$, 
a case which is amenable to calculation because of the confluence of hyperelliptic and Artin-Schreier properties.
We first prove a decomposition result about the structure of $\HdR(X)$ as a module under the actions of $F$ and $V$,
where the pieces of the decomposition are indexed by the branch points of the hyperelliptic cover.
This is the only decomposition result about the de Rham cohomology of Artin-Schreier curves that we know of,  
although the action of $V$ on $\HH^0(X, \Omega^1)$ and the action of $F$ on $\HH^1(X, \cO)$ have been studied 
for Artin-Schreier curves under less restrictive hypotheses (e.g., \cite{madden, Sullivan}). 

In the second result of the paper,  
we give a complete classification of the isomorphism classes of group schemes which occur as 
the $2$-torsion group scheme $J_X[2]$ for a hyperelliptic $k$-curve $X$ of arbitrary genus when ${\rm char}(k)=2$.  
The group schemes which occur decompose into pieces indexed by the branch points of the hyperelliptic cover, 
and we determine the Ekedahl-Oort types of these pieces.
In particular, we determine which $a$-numbers occur for the $2$-torsion group schemes of hyperelliptic $k$-curves of arbitrary genus
when ${\rm char}(k)=2$.  Before describing the result precisely, 
we note that it shows that the group scheme $J_X[2]$ depends only on some discrete 
invariants of $X$ and not on the location of the branch points or the equation of the hyperelliptic cover.
This is in sharp contrast with the case of hyperelliptic curves in odd characteristic $p$, where even the 
$p$-rank depends on the location of the branch points, \cite{Y}.

Here is some notation needed to describe the results precisely.

\begin{notation} \label{Nsetup}  
Suppose $X$ is a hyperelliptic curve defined over an algebraically closed field $k$ of characteristic $2$.
Let $B \subset \bP^1(k)$ be the set of branch points of the hyperelliptic cover $\cov: X \to \bP^1_k$.
Without loss of generality, one can suppose $\infty \not \in B$.
For $\alpha \in B$, the {\it ramification invariant} $d_\alpha$   
is the largest integer for which the higher ramification group of $\cov$ above $\alpha$ is non-trivial; 
note that $d_\alpha$ is odd. 
Then $X$ has an affine equation of the form $y^2-y=f(x)$
for some $f(x) \in k(x)$.
Furthermore, $f(x)$ can be chosen such that its partial fraction decomposition has the form:
\begin{equation} \label{eqn:pfd}
f(x)=\sum_{\alpha\in B} f_{\alpha} \big(x_\alpha\big)
\end{equation}
where  $x_\alpha:=(x-\alpha)^{-1}$ and where
$f_\alpha(x) \in xk[x^2]$ is a polynomial of degree $d_\alpha$ containing no monomials of even exponent.
In particular, the divisor of poles of $f(x)$ on $\bP_k^1$ has the form
$$
{\di}_\infty(f(x)) = \sum_{\alpha\in B} d_{\alpha} \alpha.
$$
Let $c_\alpha=(d_\alpha-1)/2$ and let $r=\#B-1$.
\end{notation}

By the Riemann-Hurwitz formula, the genus of $X$ is $g=r+\sum_{\alpha \in B} c_\alpha$.
By the Deuring-Shafarevich formula, the $p$-rank of $J_X$ equals $r$.

\begin{theorem} \label{Tmainresult}
Suppose $X$ is a hyperelliptic $k$-curve with affine equation $y^2-y=f(x)$ as described in Notation \ref{Nsetup}.
For $\alpha \in B$, consider the Artin-Schreier $k$-curve $Y_\alpha$ with affine equation $y^2-y=f_\alpha(x)$.
Let $E$ be an ordinary elliptic $k$-curve.
As a module under the actions of $F$ and $V$, the de Rham cohomology of $X$ decomposes as: 
$$
\HdR(X) \cong  \HdR(E)^{\#B-1} \oplus \bigoplus_{\alpha\in B} \HdR(Y_\alpha).
$$
\end{theorem}

As an application of Theorem \ref{Tmainresult}, we give a complete classification of the Ekedahl-Oort types 
which occur for hyperelliptic $k$-curves.
For $p=2$ and a natural number $c$, 
let $G_c$ be the unique symmetric ${\rm BT}_1$ group scheme of rank $p^{2c}$ with Ekedahl-Oort type
$[0,1,1,2,2, \ldots,  \lfloor c/2 \rfloor]$.
For example, $G_1$ is the $p$-torsion group scheme of a supersingular elliptic $k$-curve.
The group scheme $G_2$ occurs as the $p$-torsion of a supersingular non-superspecial abelian surface over $k$.
The group scheme $G_c$ is not necessarily indecomposable.
More explanation about $G_c$ is given in Sections \ref{Seotype} and \ref{SdesEO}.
The $a$-number of $X$ is $a_X := {\rm dim}_k {\rm Hom}(\alpha_p, J_X[p])$,
where $\alpha_p$ is the kernel of Frobenius on ${\mathbb G}_a$.
 
\begin{theorem} \label{Tintro}
Suppose $X$ is a hyperelliptic $k$-curve with affine equation $y^2-y=f(x)$ as described in Notation \ref{Nsetup}.
Then the $2$-torsion group scheme of the Jacobian variety of $X$ is 
$$
J_X[2] \simeq  (\ZZ/2 \oplus \mu_2)^r \oplus \bigoplus_{\alpha \in B} G_{c_{\alpha}},
$$
and the $a$-number of $X$ is $a_X=(g+1-\#\set{\alpha \in B \  \mid \ d_\alpha \equiv 1 \bmod 4})/2$.
\end{theorem}
 
Theorem \ref{Tintro} is stated without proof in \cite[3.2]{V:cycles}
in the special case that $f(x) \in k[x]$, i.e., $r=0$.

There are two interesting things about Theorem \ref{Tintro}. 
First, it shows that the Ekedahl-Oort type of $X:y^2-y=f(x)$ depends only on the orders of the poles of $f(x)$.
This is in sharp contrast with the case of hyperelliptic curves in odd characteristic, where even the 
$p$-rank depends on $f(x)$ and the location of the branch points, \cite{Y}.
Similarly, it is in contrast with the results of \cite{B}, \cite{El:bound}, \cite{John},
all of which give bounds for the $p$-rank and $a$-number of various kinds of curves
that depend strongly on the coefficients of $f(x)$. 
Likewise, preliminary calculations indicate that it is in contrast 
with the situation for Artin-Schreier curves in odd characteristic.

Secondly, Theorem \ref{Tintro} is interesting because it shows that most of the possibilities 
for the $2$-torsion group scheme do not occur for Jacobians of hyperelliptic $k$-curves when ${\rm char}(k)=2$.
Specifically, there are $2^g$ possibilities for the $2$-torsion group scheme of a $g$-dimensional abelian variety over $k$. 
We determine a subset of these, having cardinality equal to the number of partitions of $g+1$,  
and prove that the group schemes in this subset are exactly those which occur as the $2$-torsion 
$J_X[2]$ for a hyperelliptic $k$-curve $X$ of genus $g$.
Also, Theorem \ref{Tintro} gives the non-trivial bounds $(g-r)/2 \leq a_X \leq (g+1)/2$ for the $a$-number. 

An earlier non-existence result of this type can be found in \cite{Ekedahl}, where the author proved 
that a curve $X$ of genus $g > p(p-1)/2$ in characteristic $p > 0$ cannot be superspecial, and thus $a_X < g$.
There are also other recent results about Newton polygons of hyperelliptic (i.e., Artin Schreier) 
curves in characteristic $2$, including several non-existence results, \cite{SZ:02}, \cite{blache08}, \cite{blache09}, 
and counting results for curves of genus $3$ over finite fields, \cite{NS:hyper}.

Here is an outline of this paper: Section \ref{Snot} contains notation and background.  
Results on $\HH^0(X, \Omega^1)$ and the $a$-number 
are in Section \ref{Sanumber}.  Theorem \ref{Tmainresult} is with the material on the de Rham cohomology in Section \ref{Sderham}.
Section \ref{Sreseotype} contains the results about the Ekedahl-Oort type, including Theorem \ref{Tintro}.

\section{Notation} \label{Snot}

In this paper, all objects are defined over an algebraically closed field $k$ of characteristic $p > 0$ and
all curves are smooth, projective, and connected.

\subsection{The $p$-torsion group scheme}

Suppose $A$ is a principally polarized abelian variety of dimension $g$ defined over $k$.
For example, $A$ could be the Jacobian of a $k$-curve of genus $g$.
Consider the multiplication-by-$p$ morphism $[p]:A \to A$ which is a finite flat morphism of degree $p^{2g}$.
It factors as $[p]=V \circ F$.  Here $F:A \to A^{(p)}$ is the relative Frobenius morphism 
coming from the $p$-power map on the structure sheaf; it is purely inseparable of degree $p^g$.  
The Verschiebung morphism 
$V:A^{(p)} \to A$ is the dual of $F$.  

The kernel of $[p]$ is $A[p]$, the $p$-torsion of $A$, which is a quasi-polarized $BT_1$ group scheme.  
In other words, it is a quasi-polarized finite commutative group scheme annihilated by $p$, again having 
morphisms $F$ and $V$.
The rank of $A[p]$ is $p^{2g}$.
The quasi-polarization implies that $A[p]$ is symmetric.
These group schemes were classified independently by
Kraft (unpublished) \cite{Kraft} and by Oort \cite{O:strat}.
A complete description of this topic can be found in \cite{O:strat} or \cite{M:group}.

Two invariants of (the $p$-torsion of) an abelian variety are the $p$-rank and $a$-number.
The {\it $p$-rank} of $A$ is $r=\dime_{\FF_p} \Hom(\mu_p, A[p])$
where $\mu_p$ is the kernel of Frobenius on $\GG_m$.
Then $p^r$ is the cardinality of $A[p](k)$.
The {\it $a$-number} of $A$ is $a=\dime_k \Hom(\alpha_p, A[p])$ 
where $\alpha_p$ is the kernel of Frobenius on $\GG_a$.
It is well-known that $0 \leq f \leq g$ and $1 \leq a +f \leq g$.

One can describe the group scheme $A[p]$ using the theory of covariant Dieu\-donn\'e modules.
This is the dual of the contravariant theory found in \cite{Demazure}; see also \cite[A.5]{G:book}.
Briefly, 
consider the non-commutative ring $\bE=k[F,V]$ generated by semi-linear operators $F$ and $V$ with the relations 
$FV=VF=0$ and $F \lambda = \lambda^p F$ and $\lambda V=V \lambda^p$ for all $\lambda \in k$.
Let $\bE(A,B)$ denote the left ideal $\bE A+\bE B$ of $\bE$ generated by $A$ and $B$.
A deep result is that the Dieudonn\'e functor $D$ gives an equivalence of categories between $BT_1$ group schemes 
${\mathbb G}$ (with rank $p^{2g}$) and finite left $\bE$-modules $D({\mathbb G})$
(having dimension $2g$ as a $k$-vector space).
For example, the Dieudonn\'e module of an ordinary elliptic curve is 
$D(\ZZ/p \oplus \mu_p) \simeq \bE/\bE(F, 1-V) 
\oplus \bE/\bE(V, 1-F)$, \cite[Ex.\ A.5.1 \& 5.3]{G:book}.

The $p$-rank of $A[p]$ is the dimension of $V^g D({\mathbb G})$.
The $a$-number of $A[p]$ equals $g-{\rm dim}(V^2D({\mathbb G}))$ \cite[5.2.8]{LO}.

\subsection{The Ekedahl-Oort type} \label{Seotype}

The isomorphism type of a symmetric $BT_1$ group scheme ${\mathbb G}$ over $k$
can be encapsulated into combinatorial data.
This topic can be found in \cite{O:strat}.
If ${\mathbb G}$ has rank $p^{2g}$, 
then there is a {\it final filtration} $N_1 \subset N_2 \subset \cdots \subset N_{2g}$ 
of $D({\mathbb G})$ as a $k$-vector space
which is stable under the action of $V$ and $F^{-1}$ such that $i={\rm dim}(N_i)$.
If $w$ is a word in $V$ and $F^{-1}$, then $wD({\mathbb G})$ is an object in the filtration.
In particular, $N_g = V D({\mathbb G}) =F^{-1}(0)$.
If ${\mathbb G}$ is quasi-polarized, then there is a sympletic form on $D({\mathbb G})$
and $N_{2g-i}$ and $N_i$ are orthogonal under the symplectic pairing.

The {\it Ekedahl-Oort type} of ${\mathbb G}$, also called the {\it final type},
is $\nu=[\nu_1, \ldots, \nu_r]$ where ${\nu_i}={\rm dim}(V(N_i))$.
The Ekedahl-Oort type of ${\mathbb G}$ is canonical, even if the final filtration is not.
There is a restriction $\nu_i \leq \nu_{i+1} \leq \nu_i +1$ on the final type.
All sequences satisfying this restriction occur.   
This implies that there are $2^g$ isomorphism classes of symmetric $BT_1$ group schemes of rank $p^{2g}$.    
The $p$-rank is ${\rm max}\set{i \mid \nu_i=i}$ and the $a$-number equals $g-\nu_g$.

\subsection{The de Rham cohomology} \label{Sdefderham}

Suppose $X$ is a $k$-curve of genus $g$
and recall the definition of the non-commutative ring $\bE=k[F,V]$ from Section \ref{Sdefderham}.
By \cite[Section 5]{Oda}, 
there is an isomorphism of $\bE$-modules between the Dieudonn\'e module 
$D(J_X[p])$ and the de Rham cohomology $\HdR(X)$.  
In particular, $\ker(F) = \HH^0(X, \Omega^1)={\rm im}(V)$.
Recall that $\dim_k \HdR(X) = 2g$.

In \cite[Section 5]{Oda}, there is the following description of $\HdR(X)$. 
Let $\cU = \set{U_i}$ be a covering of $X$ by affine open subvarieties and 
let $U_{ij} := U_i \cap U_j$ and $U_{ijk} := U_i \cap U_j \cap U_k$.
For a sheaf $\cF$ on $X$, let
\begin{eqnarray*}
\CC^0(\cU, \cF) &:=& \set{\kappa = (\kappa_i)_i \mid \kappa_i\in \Gamma(U_i, \cF)},\\
\CC^1(\cU, \cF) &:=& \set{\phi = (\phi_{ij})_{i<j} \mid \phi_{ij}\in \Gamma(U_{ij}, \cF)},\\
\CC^2(\cU, \cF) &:=& \set{\psi = (\psi_{ijk})_{i<j<k} \mid \psi_{ijk}\in \Gamma(U_{ijk}, \cF)}.
\end{eqnarray*}
For convenience, let $\phi_{ii} := 0$ for any $\phi\in \CC^1(\cU, \cF)$.
There are coboundary operators $\delta: \CC^0(\cU, \cF) \to \CC^1(\cU, \cF)$ defined by
$(\delta \kappa)_{i<j} = \kappa_i - \kappa_j$; and $\delta: \CC^1(\cU, \cF) \to \CC^2(\cU, \cF)$ by
$(\delta \phi)_{i<j<k} = \phi_{ij} - \phi_{ik} + \phi_{jk}$.
All other maps are applied to $\CC^m(\cU, \cF)$ elementwise, e.g., $(F\phi)_i := F\phi_i$.

The de Rham cocycles are defined by
$$
\ZdR(\cU) := \set{(\phi, \omega)\in \CC^1(\cU, \cO) \times \CC^0(\cU, \Omega^1) \mid
\delta \phi=0,\ d\phi = \delta\omega},
$$
that is, $\phi_{ij}-\phi_{ik} + \phi_{jk} = 0$ and $d\phi_{ij} = \omega_i - \omega_j$ 
for all indices $i < j < k$.  
The de Rham coboundaries are defined by
\begin{equation*}
\BdR(\cU) := \set{(\delta \kappa, d\kappa)\in \ZdR(\cU) \mid \kappa\in C^0(\cU, \cO)}.
\end{equation*}
Finally,
$$
\HdR(X) \cong \HdR(\cU) := \ZdR(\cU) / \BdR(\cU).
$$


There is an injective homomorphism $\lambda:\HH^0(X, \Omega^1) \to \HdR(X)$
denoted informally by $\omega \mapsto (0, \omega)$ where the second coordinate 
is defined by $\omega_i=\omega \vert_{U_i}$.
This map is well-defined since $d(0)= \omega \vert_{U_i} - \omega \vert_{U_j}= (\delta \omega)_{i<j}$.
It is injective because, if $(0, \omega_1) \equiv (0, \omega_2) \bmod{\BdR(\cU)}$,
then $\omega_1 - \omega_2 = d\kappa$ where $\kappa \in C^0(\cU, \cO)$ is such that 
$\delta \kappa=0$;
thus $\kappa\in \CH0(\cU, \cO) \simeq k$ is a constant function on $X$ and so $\omega_1 - \omega_2 = 0$.

There is another homomorphism $\gamma: \HdR(X) \to \CH^1(X, \cO)$ sending
the cohomology class of $(\phi,\omega)$ to the cohomology class of $\phi$. 
The choice of cocycle $(\phi, \omega)$ does not matter, since 
the coboundary conditions on $\HdR(X)$ and $\CH^1(X, \cO)$ are compatible.
The homomorphisms $\lambda$ and $\gamma$ fit into a short exact sequence
\begin{equation*}
0 \to \HH^0(X, \Omega^1) \xrightarrow{\lambda} \HdR(X)
\xrightarrow{\gamma} \HH^1(X,\cO) \to 0.
\end{equation*}
In Subsections \ref{sub:auxiliary} and \ref{sub:derham}, 
we construct a suitable section $\sigma: \HH^1(X, \cO)\to \HdR(X)$ of $\gamma$
when $X$ is a hyperelliptic $k$-curve with ${\rm char}(k)=2$.

\subsection{Frobenius and Verschiebung}
\label{sub:fv}

The Frobenius and Verschiebung operators $F$ and $V$ act on $\HdR(X)$ as follows:
$$
F(f, \omega) := (f^p, 0)
\quad\text{and}\quad
V(f, \omega) := (0, \cC(\omega)),
$$
where $\cC$ is the Cartier operator \cite{Cartier} on the sheaf $\Omega^1$.
The operator $F$ is $p$-linear and $V$ is $p^{-1}$-linear.

The three principal properties of the Cartier operator are that
it annihilates exact differentials, preserves logarithmic ones, and is $p^{-1}$-linear.
The Cartier operator can be computed as follows.
Let $x\in k(X)$ be an element which forms a $p$-basis of $k(X)$ over $k(X)^p$,
i.e., such that every $z\in k(X)$ can be written as
$$
z := z_0^p + z_1^p x + \cdots + z_{p-1}^p x^{p-1}
$$
for uniquely determined $z_0, \ldots, z_{p-1} \in k(X)$.
Then
$$
\cC(z \dx/x) := z_{0} \dx/x.
$$

\section{Results about holomorphic $1$-forms and the $a$-number} \label{Sanumber}

\subsection{Equations}\label{Sequations}

We specialize to the case that 
$k$ is an algebraically closed field of characteristic $p=2$.
Let $X$ be a $k$-curve of genus $g$ which is hyperelliptic, in other words, for which 
there exists a degree two cover $\cov: X \to \bP^1$.

Let $B \subset \bP^1(k)$ denote the set of branch points of $\cov$.
After a fractional linear transformation, one can suppose that $0\in B$ and $\infty \notin B$.
For $\alpha \in B$, let $P_\alpha := \cov^{-1}(\alpha) \in X(k)$ be the ramification point above $\alpha$. 
Also consider the divisor $D_\infty := \pi^{-1}(\infty)$ on $X$.
Let $B_\infty := B\cup\set\infty$, $B':=B-\set{0}$ and $r := \#B - 1$. 

For $\alpha \in B$, let $x_\alpha=(x-\alpha)^{-1}$
and let $d_\alpha$ be the {\it ramification invariant} at $\alpha$, namely
the largest integer for which the higher ramification group of $\cov$ above $\alpha$ is non-trivial.
By \cite[Prop.\ III.7.8]{sti}, $d_\alpha$ is odd. 
Let $c_\alpha:=(d_\alpha-1)/2$.

The cover $\cov$ is given by an affine equation of the form
\begin{equation*}
y^2-y=f(x)
\end{equation*}
for some non-constant rational function $f(x) \in k(x)$.
After a change of variables of the form $y \mapsto y+\epsilon$, 
one can suppose that the partial fraction decomposition of $f(x)$ has the form:
\begin{equation*}
f(x)=\sum_{\alpha\in B} f_{\alpha} \big(x_\alpha\big),
\end{equation*}
where $f_\alpha(x) \in xk[x^2]$ is a polynomial of degree $d_\alpha$ containing no monomials of even exponent.
In particular, the divisor of poles of $f(x)$ on $\bP^1$ has the form
$$
{\di}_\infty(f(x)) = \sum_{\alpha\in B} d_{\alpha} \alpha.
$$

By the Riemann-Hurwitz formula \cite[IV, Prop.\ 4]{Se:lf}, the genus of $X$ satisfies 
$$
2g+2=\sum_{\alpha\in B} (d_\alpha + 1).
$$
By the Deuring-Shafarevich formula \cite[Cor.\ 1.8]{Crew}, the $p$-rank of $X$ is $r$.
Note that $g=r+\sum_{\alpha \in B} c_\alpha$.
The implication of these formulae is that, 
for a given genus $g$ (and $p$-rank $r$), there is an additional discrete invariant of $X$, 
namely a partition of $2g+2$ into ($r+1$) even integers $d_\alpha + 1$.
In Section \ref{Sfinalfil}, we show that the Ekedahl-Oort type of $X$ depends only on this discrete invariant.

\subsection{The space $\HH^0(X, \Omega^1)$} \label{SHzero}

For an integer $j$ and for $\alpha \in B$, consider the following $1$-forms on $X$: 
$$
\omega_{\alpha,j} := x_\alpha^{j-1} \dx_\alpha.
$$
Note that $\omega_{\alpha,j} = - (x-\alpha)^{-j-1} dx$
and, if $\alpha \in B'$, then $\omega_{\alpha,0}-\omega_{0,0}=-\alpha dx/x(x-\alpha)$.

The following lemma is a variation of a special case of \cite[Lemma 1(c)]{Sullivan}. 

\begin{lemma}\label{L1formbasis}
A basis for $\HH^0(X, \Omega^1)$ is given by the $1$-forms 
$\omega_{\alpha,j}$ for $\alpha \in B$ and $1 \leq j \leq c_\alpha$
and $\omega_{\alpha, 0} - \omega_{0,0}$ for $\alpha\in B'$.
\end{lemma}

\begin{proof}
For $\alpha \in B$, one can calculate the following divisors on $X$: $\di(x_\alpha)=D_\infty - 2P_\alpha$ and 
\begin{equation}\label{eqn:divdxalpha}
\di(dx_\alpha)= (d_\alpha-3) P_\alpha + \sum_{\beta \in B-\set\alpha}(d_\beta +1)P_\beta
\end{equation}
and
\begin{equation}\label{eqn:divomega}
\di(\omega_{\alpha,j})=
2(c_\alpha-j)P_\alpha + (j-1) D_\infty + \sum_{\beta \in B-\set{\alpha}}(d_\beta +1)P_\beta.
\end{equation}
Thus $\omega_{\alpha, j}$ is holomorphic for $1 \leq j \leq c_\alpha$.
Also, $\omega_{\alpha,0}-\omega_{0,0}$ is holomorphic for $\alpha \in B'$ because
\begin{eqnarray*}
\di(\omega_{\alpha, 0}-\omega_{0,0})=2c_\alpha P_\alpha + 2c_0 P_0 +\sum_{\beta \in B-\set{0,\alpha}}(d_\beta +1)P_\beta.
\end{eqnarray*}
This set of holomorphic differentials of $X$ is linearly independent because 
the corresponding set of divisors is linearly independent over $\ZZ$.
It forms a basis since the set has cardinality $r + \sum_{\alpha \in B} c_\alpha=g$.
\end{proof}

\begin{lemma}\label{Lcartier1}
If $\alpha\in B$, then
$$
\cC(\omega_{\alpha, j}) =
\begin{cases}
\omega_{\alpha, j/2} & \text{if $j$ is even},\\
0 & \text{if $j$ is odd}.
\end{cases}
$$
In particular, $\cC(\omega_{\alpha,0} - \omega_{0,0}) = \omega_{\alpha,0} - \omega_{0,0}$
for all $\alpha\in B'$.
\end{lemma}
\begin{proof}
Using the properties of the Cartier operator found in Section \ref{sub:fv}, 
one computes when $j$ is even that
$$\cC\left(x_\alpha^{j-1}\dx_\alpha\right) = 
x_\alpha^{j/2} \cC\left( dx_\alpha / x_\alpha \right)=
x_\alpha^{j/2-1}\dx_\alpha,$$
and when $j$ is odd that 
$$\cC\left(x_\alpha^{j-1}\dx_\alpha\right) = x_\alpha^{(j-1)/2} \cC(dx_\alpha) =0.$$
\end{proof}

For $\alpha \in B'$, let $W'_{\alpha,\sss} := \langle \omega_{\alpha, 0}-\omega_{0,0}\rangle,$
and for $\alpha \in B$, let $W'_{\alpha,\nil} := \langle \omega_{\alpha, j} \mid 1\leq j \leq c_\alpha \rangle$, 
where $\langle \cdot \rangle$ denotes the $k$-span.
These subspaces are invariant under the Cartier operator by Lemma \ref{Lcartier1}.

\begin{lemma}\label{L1form}
The subspaces $W'_{\alpha,\sss}$ and $W'_{\alpha,\nil}$ of $\HH^0(X, \Omega^1)$ 
are stable under the action of Verschiebung for each $\alpha \in B$.
There is an isomorphism of $V$-modules:
$$\HH^0(X, \Omega^1) \simeq
\bigoplus_{\alpha\in B'} W'_{\alpha,\sss} \oplus \bigoplus_{\alpha\in B} W'_{\alpha,\nil}.$$
\end{lemma}

\begin{proof}
This follows immediately from Lemmas \ref{L1formbasis} and \ref{Lcartier1}.
\end{proof}

\subsection{Application: the $a$-number}

\begin{proposition} \label{Panumber}
Suppose $X$ is a hyperelliptic $k$-curve with affine equation $y^2-y=f(x)$ as described in Notation \ref{Nsetup}.
If ${\rm div}_\infty(f(x))=\sum_{\alpha \in B} d_\alpha \alpha$ is the divisor of poles of $f(x)$ on $\bP^1$, 
then the $a$-number of $X$ is 
$$
a_X= \frac{g+1-\#\set{\alpha\in B \mid d_\alpha \equiv 1 \bmod 4}}{2}.
$$
\end{proposition}

\begin{proof}
The $a$-number of $\mathbb G=J_X[2]$ is $a_X=g-{\rm dim}(V^2D({\mathbb G}))$ \cite[5.2.8]{LO}.
The action of $V$ on $VD(\mathbb G)$ is the same as the action of the 
Cartier operator $\cC$ on $\HH^0(X, \Omega^1)$.
So $a_X$ equals the dimension of the kernel of $\cC$ on $\HH^0(X, \Omega^1)$.
By Lemma \ref{Lcartier1}, the kernel of $\cC$ on $\HH^0(X, \Omega^1)$ is spanned by
$\omega_{\alpha,j}$ for $\alpha \in B$ and $j$ odd with $1 \leq j \leq c_\alpha = (d_\alpha - 1)/2$.
Thus the contribution to the $a$-number from each $\alpha \in B$ is $\lfloor (d_\alpha + 1)/4 \rfloor$.
In other words, if $d_\alpha \equiv 1 \bmod 4$, the contribution is $(d_\alpha-1)/4$ and 
if $d_\alpha \equiv 3 \bmod 4$, the contribution is $(d_\alpha+1)/4$.
Since $g + 1 = \sum_{\alpha \in B} (d_\alpha + 1)/2$, this yields
$$
2a_X = (g+1) - \# \set{\alpha \in B \mid d_\alpha \equiv 1 \bmod 4 }.
$$
\end{proof}


\subsection{Examples with large $p$-rank}
Let $A$ be a principally polarized abelian variety over $k$ with dimension $g$ and $p$-rank $r$.
If $r=g$, then $A[p] \simeq (\ZZ/p \oplus \mu_p)^g$ and the $a$-number is $a=0$.
If $r=g-1$ then $A[p] \simeq (\ZZ/p \oplus \mu_p)^{g-1} \oplus E[p]$ where $E$ is a supersingular elliptic curve 
and the $a$-number is $a=1$.  
So the first case where $A[p]$ and $a$ are not determined by the $p$-rank is when $r=g-2$.

\begin{example} Let $p=2$ and $g \geq 2$.  
There are two possibilities for the $p$-torsion group scheme of a principally polarized abelian variety over $k$
with dimension $g$ and $p$-rank $g-2$.
Both of these occur as the $2$-torsion group scheme $J_X[2]$
of the Jacobian of a hyperelliptic $k$-curve $X$ of genus $g$.   
\end{example}

\begin{proof}
If $A$ is a principally polarized abelian variety over $k$ with dimension $g$ and $p$-rank $g-2$, then 
$A[p] \simeq (\mu_p \oplus \ZZ/p)^{g-2} \oplus {\mathbb G}$ where ${\mathbb G}$ is 
isomorphic to the $p$-torsion group scheme of an abelian surface $Z$ with $p$-rank $0$.
There are two possibilities for ${\mathbb G}$, according to whether $Z$ is superspecial or merely supersingular:
$(G_1)^2$, where $G_1$ denotes the $p$-torsion group scheme of a supersingular elliptic $k$-curve;
and $G_2$, which is isomorphic to the quotient of $(G_1)^2$ by a generic $\alpha_p$ subgroup scheme.

To prove the second claim, consider the two possibilities for a partition of $2g+2$ into $r+1=g-1$ even integers, namely 
(A) $\set{2,2, \ldots, 2,4,4}$ or (B) $\set{2,2, \ldots, 2,2,6}$.
In case (A), consider $f(x) \in k(x)$ with $g-1$ poles, such that $0$ and $1$ are poles of order $3$ 
and the other poles are simple.
In case (B), consider $f(x) \in k(x)$ with $g-1$ poles, such that $0$ is a pole of order $5$ and the other poles are simple.
The kernel of the Cartier operator on $\HH^0(X, \Omega^1)$ is spanned by $dx/x^2$ and $dx/(x-1)^2$ in case (A) 
and by $dx/x^2$ in case (B).
Thus the $a$-number equals $2$ in case (A) and equals $1$ in case (B).
In both cases, this completely determines the group scheme.
Namely, the group scheme $J_X[2]$ is isomorphic to 
$(\ZZ/2 \oplus \mu_2)^{g-2} \oplus (G_1)^2$ in case (A) 
and to $(\ZZ/2 \oplus \mu_2)^{g-2} \oplus G_2$ in case (B).
\end{proof}

For $g \geq 3$ and $r \leq g-3$, the action of $V$ on $\HH^0(X, \Omega^1)$ 
(and, in particular, the value of the $a$-number) is not sufficient
to determine the isomorphism class of the group scheme $J_X[2]$.
To determine this group scheme, in the next section we study the $\bE$-module structure of 
$\HH^1_{\rm dR}(X)$. 

\section{Results on the de Rham cohomology} \label{Sderham}
\subsection{An open covering}

Let $V'=\bP^1 - B_\infty$ and $U'=\cov^{-1}(V')=X-\cov^{-1}(B_\infty)$.
For $\alpha \in B_\infty$, let $V_\alpha=V' \cup \set{\alpha}$ and 
$U_\alpha=U' \cup \{\cov^{-1}(\alpha)\}$.
The collection $\cU := \set{U_\alpha \mid \alpha \in B_\infty}$ is a cover of $X$ by open affine subvarieties.
By construction, if $\alpha, \beta \in B_\infty$ are distinct, then
$V_{\alpha\beta}:=V_{\alpha} \cap V_{\beta}=V'$ and 
$U_{\alpha\beta} := U_{\alpha} \cap U_{\beta} = U'$.
In particular, the subvarieties $U_{\alpha \beta}$ do not depend on the choice of
$\alpha$ and $\beta$.

For a sheaf $\cF$, let $\CZ^1(\cU, \cF)$ and $\CB^1(\cU, \cF)$ denote the closed cocycles and coboundaries 
of $\cF$ with respect to $\cU$.  Recall the definition of the non-commutative ring $\bE=k[F,V]$ 
and the notation about $\HdR(X)$ from Section \ref{Sdefderham}. 
In this section, we compute $\HH^1(X, \cO) \simeq \HH^1(\cU, \cO)$ and $\HdR(X) \simeq \HdR(\cU)$ with respect
to the open covering $\cU$ of $X$.

\subsection{Defining components}

Given a sheaf $\cF$ and a cocycle $\phi \in \CZ^1(\cU, \cF)$, 
consider its components $\phi_{\alpha \infty} \in \Gamma(U', \cF)$ for $\alpha \in B$.
We call $\{\phi_{\alpha \infty} \mid \alpha \in B\}$ the set of {\em defining components} of $\phi$.
The reason is that the remaining components of $\phi$ are determined by the coboundary condition
$\phi_{\alpha\beta} = \phi_{\alpha\infty} - \phi_{\beta\infty}$.
A collection of sections $\{\phi_{\alpha \infty} \in \Gamma(U', \cF) \mid \alpha \in B\}$ 
determines a unique closed cocycle $\phi \in \CZ^1(\cU, \cF)$. 
Thus,
\begin{equation} \label{Edirectsum}
\CZ^1(\cU, \cF) \cong \bigoplus_{\alpha \in B} \Gamma(U', \cF).
\end{equation}

For $\beta\in B$, consider the natural $k$-linear map
$$
\varphi_\beta: \Gamma(U', \cO) \to \CZ^1(\cU, \cO),
$$
whose defining components for $\alpha \in B$ are
\begin{equation*}
(\varphi_\beta(h))_{\alpha\infty} :=
\begin{cases}
h & \text{if $\alpha = \beta$},\\
0 & \text{otherwise.}
\end{cases}
\end{equation*}
Also, consider the $k$-linear map $\varphi_\infty:\Gamma(U', \cO) \to \CZ^1(\cU, \cO)$ defined by: 
\begin{equation*}
(\varphi_\infty(h))_{\alpha\infty} := -h
\quad\text{for all $\alpha\in B$}.
\end{equation*}
Observe that if $h\in  \Gamma(U', \cO)$, then
\begin{equation} \label{Esum}
\sum_{\beta\in B_\infty} \varphi_\beta(h) = 0.
\end{equation}
For $\beta \in B_\infty$, consider the natural $k$-linear map
$$
\psi_\beta: \Gamma(U_{\beta}, \cO) \to \CC^0(\cU, \cO)
$$
given for $\alpha \in B_\infty$ by
\begin{equation}
(\psi_\beta(h))_\alpha := 
\begin{cases}
h & \text{if $\alpha = \beta$},\\
0 & \text{otherwise.}
\end{cases}
\end{equation}

It is straightforward to verify the next lemma.
\begin{lemma} \label{Lcobound}
Suppose $\beta\in B_\infty$
and $h \in \Gamma(U_\beta, \cO)$ (i.e., $h$ is regular at $P_\beta$ if $\beta\neq \infty$
and $h$ is regular at the two points in the support of $D_\infty$ if $\beta=\infty$).  
Then $\varphi_\beta(h|_{U'}) = \delta\psi_\beta(h)$ is a coboundary.
\end{lemma}

\subsection{The space $\HH^1(X, \cO)$}

In this section, we find an $F$-module decomposition of $\HH^1(X, \cO) \simeq \HH^1(\cU, \cO)$.  
The results could be deduced from Section \ref{SHzero} using the duality between 
$\HH^1(X, \cO)$ and $\HH^0(X, \Omega^1)$.
Instead, we take a direct approach, because an explicit description of $\HH^1(X, \cO)$
is helpful for studying $\HdR(X)$ in Section \ref{sub:derham}.
The following lemmas will be useful.

\begin{lemma} \label{Ldivisor}
For $\alpha\in B$ and $j \in \ZZ$, the divisor of poles on $X$ of the function $yx_\alpha^{-j}=y(x-\alpha)^j$ is: 
$$\di_\infty(y(x-\alpha)^{j}) = \max(d_\alpha-2j, 0) P_\alpha + \max(j, 0)D_\infty
+ \sum_{\beta\in B-\set\alpha}d_\beta P_\beta.$$
\end{lemma}

\begin{proof}
Indeed, $\di(x-\alpha) =  2 P_\alpha - D_\infty$ for $\alpha\in B$ and
$$\di_\infty(y) =  \sum_{\beta \in B} d_\beta P_\beta.$$
\end{proof}


Lemma \ref{Ldivisor} implies that $y(x-\alpha)^j \in \Gamma(U', \cO)$ for all $\alpha \in B$ and $j \in \ZZ$.

\begin{lemma}\label{lem:basicfacts}
With notation as above:
\begin{enumerate}
\item[(i)]
$\CZ^1(\cU, \cO) = \gen{ \varphi_\beta((x-\alpha)^j), \varphi_\beta(y(x-\alpha)^j) \mid \alpha, \beta \in B, \ j\in \Z}$.

\item[(ii)]
If $\alpha \in B$, then $\gen{\varphi_\alpha(y(x-\beta)^j) \mid j \geq 0} = \gen{\varphi_{\alpha}(y(x-\alpha)^j) \mid j \geq 0}$
as subspaces of $\CZ^1(\cU, \cO)$ for each $\beta \in B$.
\end{enumerate}
\end{lemma}

\begin{proof}
\begin{enumerate}
\item[(i)]
This is immediate from Equation \eqref{Edirectsum} because
$$\CZ^1(\cU, \cO) = \bigoplus_{\beta \in B} \gen{\varphi_\beta(h) \mid h \in \Gamma(U', \cO)}.$$

\item[(ii)]
Both are equal to the subspace $\set{\varphi_\alpha(yh(x)) \mid h(x) \in k[x]}$.
\end{enumerate}
\end{proof}

\begin{lemma} \label{lem:boundary}
Let $\alpha \in B$ and $j \in \ZZ$.  Then:
\begin{enumerate}
\item[(i)]
$\varphi_\beta((x-\alpha)^j) \in \CB^1(\cU,\cO)$ for all $\beta \in B_\infty$.
\item[(ii)]
$\varphi_{\alpha}(y(x-\alpha)^j) \in \CB^1(\cU,\cO)$ if $j>c_\alpha$.
\item[(iii)]
$\varphi_\infty(y(x-\alpha)^j) \in \CB^1(\cU,\cO)$ if $j \leq 0$.
\end{enumerate}
\end{lemma}

\begin{proof}
\begin{enumerate}
\item[(i)]
Suppose that $\beta\in B$.
If $\beta \neq \alpha$ or if $j \geq 0$, then $(x-\alpha)^j$ is regular at $P_\beta$ and so 
$\varphi_\beta((x-\alpha)^j)\in \CB^1(\cU, \cO)$ by Lemma \ref{Lcobound}. 
For $j \geq 0$, it follows from this and Equation \eqref{Esum} that the cocycle 
$\varphi_\infty((x-\alpha)^j) = -\sum_{\beta\in B} \varphi_\beta((x-\alpha)^j)$
is a coboundary.
If $j<0$, then $\varphi_\infty((x-\alpha)^j) \in \CB^1(\cU,\cO)$ by Lemma \ref{Lcobound}.

Finally, if $\beta =\alpha \neq \infty$ and $j<0$, then $(x-\alpha)^j \in \Gamma(U_\gamma, \cO)$ for
all $\gamma\in B_\infty-\set\alpha$.  By Equation \eqref{Esum}, 
\begin{equation}\label{eqn:alphaalpha}
\varphi_\alpha((x-\alpha)^j) = 
- \sum_{\gamma\in B_\infty-\set\alpha} \varphi_\gamma\bigl((x-\alpha)^j\bigr)
= 
- \sum_{\gamma\in B_\infty - \set\alpha} \delta\psi_ \gamma((x-\alpha)^{j}), 
\end{equation}
which is a coboundary.

\item[(ii)]
If $j > c_\alpha$, then $y(x-\alpha)^j \in \Gamma(U_\alpha, \cO)$ and 
$\varphi_\alpha(y(x-\alpha)^j) = \delta\psi_\alpha(y(x-\alpha)^j)$.
\item[(iii)]
If $j\leq 0$, then $y(x-\alpha)^j \in \Gamma(U_\infty, \cO)$,
and $\varphi_\infty(y(x-\alpha)^j) = \delta\psi_\infty(y(x-\alpha)^j)$.
\end{enumerate}
\end{proof}

Consider the cocycles $\phi_{\alpha, j} \in \CZ^1(\cU, \cO)$ for
$\alpha \in B$ and $j\in \Z$ defined by
$$
\phi_{\alpha, j} := \varphi_\alpha(y (x-\alpha)^{j}).
$$
Given $\phi\in\CZ^1(\cU, \cO)$, denote by $\tilde\phi$ the cohomology
class of $\phi$ in $\HH^1(\cU, \cO)$.
For $\alpha \in B_\infty$, define the map
$$
\tilde\varphi_\alpha: \Gamma(U',\cO) \to \HH^1(\cU, \cO), 
\quad
f\mapsto \varphi_\alpha(f) \bmod{\CB^1(\cU, \cO)}.
$$

The following lemma is a variant of a special case of \cite[Lemma 6]{madden}.

\begin{lemma}\label{Lnextbasis}
A basis for $\HH^1(\cU, \cO)$ is given by the cohomology classes 
$\tilde\phi_{\alpha,j}$ for $\alpha \in B$ and $1 \leq j \leq c_\alpha$,
and $\tilde\phi_{\alpha, 0}$ for $\alpha\in B'$.
\end{lemma}

\begin{proof}
The set of cohomology classes $S=\{\tilde\phi_{\alpha,j} \mid \alpha \in B, \ 1 \leq j \leq c_\alpha\} \cup 
\{\tilde\phi_{\alpha, 0} \mid \alpha\in B'\}$ has cardinality $r + \sum_{\alpha\in B} c_\alpha = g$.
By Lemmas \ref{lem:basicfacts}(i) and \ref{lem:boundary}(i), it suffices to show
that $\varphi_\beta(y(x-\alpha)^j)$ is in the span of $S$ for $\alpha, \beta \in B$ and $j \in \ZZ$.
By Lemmas \ref{lem:basicfacts}(ii) and \ref{lem:boundary}(ii), it suffices to show that 
the span of $S$ contains $\tilde\phi_{0,0}$ and $\tilde\varphi_\beta(y(x-\alpha)^{-j})$
for $\alpha, \beta\in B$ and $j>0$.

The cocycle $\varphi_\infty(y)$ is a coboundary by Lemmas \ref{Lcobound} and \ref{Ldivisor}.
Using this and Equation \eqref{Esum}, one computes in $\HH^1(\cU, \cO)$ that  
$$
\tilde\phi_{0, 0} = \tilde\varphi_0(y) +\tilde\varphi_\infty(y)
= -\sum_{\beta\in B'} \tilde\varphi_\beta(y)
= -\sum_{\beta\in B'} \tilde\phi_{\beta, 0},
$$
which is in the span of $S$.

Now consider $\tilde\varphi_\beta(y(x-\alpha)^{-j})$ for $\alpha, \beta\in B$ and $j>0$.
If $r=0$, then this cocycle is a coboundary by Equation \eqref{Esum} and Lemma \ref{lem:boundary}(iii).
Let $r>0$ and first suppose that $\alpha\neq\beta$.
After a change of coordinates, one can assume without loss of generality
that $\alpha=0$ and $\beta=1$.  For $i \geq 1$, let $b_i \in k$ be such that:
$$
\sum_{i \geq 0} b_i x^i :=  \left(\frac{1-x^{c_0+1}}{x-1}\right)^j.
$$
Then
\begin{eqnarray*}
\frac{y}{(x-1)^j} - \sum_{i \geq 0} b_i y x^i 
&=& \frac{y(1-(1-x^{c_0+1})^j)}{(x-1)^j}\\
&=& \frac{yx^{c_0+1} h(x)}{(x-1)^j}\\
\end{eqnarray*}
for some polynomial $h(x) \in k[x]$. 
The latter function is regular at $P_0$ by Lemma \ref{Ldivisor}.
Therefore, the cocycle
$\phi := \varphi_0({y x^{c_0+1} h(x)}/{(x-1)^j})$
is a coboundary by Lemma \ref{lem:boundary}(ii). 
Using this and Lemma \ref{lem:boundary}(ii) yields that
\begin{eqnarray*}
\tilde\varphi_0(y(x-1)^{-j}) 
&=& \sum_{i \geq 0} b_i \tilde\varphi_0(y x^i)
+ \tilde\varphi_0({y x^{c_0+1} h(x)}/{(x-1)^j})\\
&=& \sum_{i=0}^{c_0} b_i \tilde\phi_{0,i}, 
\end{eqnarray*}
which is in the span of $S$.
(In the general case, the cohomology class $\tilde \varphi_\beta(y(x-\alpha)^{-j})$ can be expressed 
using the Laurent expansion of $1/(x-\alpha)^j$ at $\beta$ truncated at degree $c_\alpha$.)

If $\alpha=\beta$ and $j >0$, one can add the coboundary
$\varphi_\infty(y(x-\alpha)^{-j})$ to $\varphi_\alpha(y(x-\alpha)^{-j})$ and use Equation \eqref{Esum} to see that
$$
\tilde\varphi_\alpha(y (x-\alpha)^{-j}) 
= -\sum_{\gamma\in B-\set\alpha}\tilde\varphi_\gamma(y (x-\alpha)^{-j}),
$$
which reduces the proof to the previous case.
\end{proof}

The next lemma is important for describing the $F$-module structure of $\HH^1(\cU, \cO)$.

\begin{lemma}\label{LactionFH1}
If $\alpha\in B$ and $j\geq 0$, then
$$
F\tilde\phi_{\alpha, j} =
\begin{cases}
\tilde\phi_{\alpha, 2j} & \textrm{if $2j \leq c_\alpha$},\\
0 & \textrm{otherwise.}
\end{cases}
$$
\end{lemma}
\begin{proof}
Since $(F\phi_{\alpha, j})_{\beta\gamma} = (\phi_{\alpha, j})_{\beta\gamma}^2$, one computes that
\begin{eqnarray*}
(y (x-\alpha)^{j})^2
&=& (y + f(x)) (x-\alpha)^{2j}\\
&=& y (x-\alpha)^{2j} + f(x) (x-\alpha)^{2j}.
\end{eqnarray*}
The statement follows from the definition of $\tilde \phi_{\alpha, j}$ and Lemma \ref{lem:boundary}(i).
\end{proof}

Now define
\begin{eqnarray*}
W''_{\alpha,\sss} &:=& \gen{\tilde\phi_{\alpha,0}}
\quad\textrm{for $\alpha \in B'$, and}\\
W''_{\alpha,\nil} &:=& \gen{\tilde\phi_{\alpha,j} \mid 1\leq j \leq c_\alpha}
\quad\textrm{for $\alpha \in B$}.
\end{eqnarray*}

\begin{lemma}\label{H1decomposition}
The subspaces ${W}''_{\alpha, \sss}$ and ${W}''_{\alpha, \nil}$ of $\HH^1(\cU, \cO)$
are stable under the action of Frobenius for each $\alpha \in B$.
There is an isomorphism of $F$-modules:
$$
\HH^1(\cU, \cO) \simeq
\bigoplus_{\alpha\in B'} {W}''_{\alpha,\sss} 
\oplus
\bigoplus_{\alpha\in B} {W}''_{\alpha,\nil}.
$$
\end{lemma}
\begin{proof}
This follows immediately from Lemmas \ref{Lnextbasis} and \ref{LactionFH1}.
\end{proof}

\subsection{Auxiliary map}\label{sub:auxiliary}

The next goal is to define a section $\sigma: \HH^1(X, \cO) \to \HdR(X)$.
To do this, the first step will be to define a homomorphism 
$\rho:\CZ^1(\cU, \cO) \to \CC^0(\cU, \Omega^1)$ by
defining its components $\rho_\alpha: \CZ^1(\cU, \cO) \to \Gamma(U_\beta, \Omega^1)$ for $\alpha, \beta \in B$.
Given $\phi \in\CZ^1(\cU, \cO)$ and $\alpha \in B$, 
the idea is to separate $d\phi$ into two parts: the first part will be  
regular at $P_\alpha$ and thus belong to $\Gamma(U_\alpha, \Omega^1)$;
the second part will be regular away from $P_\alpha$
and hence belong to $\Gamma(U_\beta, \Omega^1)$
for every $\beta \neq\alpha$.

\begin{notation} \label{Ntruncate}
Define the {\em truncation} operator $\OOp_{\geq i} :k[x,x^{-1}] \to k[x,x^{-1}]$ by
$$
\OOp_{\geq i}\Big(\sum_{j} a_j x^j \Big) := \sum_{j \geq i} a_j x^{j}.
$$
Operators $\OOp_{> i},\OOp_{\leq i},\OOp_{< i}:k[x,x^{-1}] \to k[x,x^{-1}]$ can be defined analogously. 
These operators can also be defined on $k[x_\alpha, x_\alpha^{-1}]$.
To clarify some ambiguity in notation, if $m(x_\alpha) \in k[x_\alpha, x_\alpha^{-1}]$,
then let $\OOp_{\geq i}(m(x_\alpha))$ denote $\OOp_{\geq i}(m(x))|_{x=x_\alpha}$.
\end{notation}

Recall that $x_\alpha := \xalpha$, and so $\phi_{\alpha,j} = \varphi_\alpha(yx_\alpha^{-j})$.
Then
\begin{equation} \label{dyxj}
d(yx_\alpha^{-j}) = - j x_\alpha^{-j-1} y \dx_\alpha + x_\alpha^{-j} dy. 
\end{equation}
Using partial fractions and the fact that  $\dy = -d(f(x))$, one sees that
\begin{equation} \label{dyparfrac} 
\dy = - \sum_{\beta\in B} f_\beta'(x_\beta) \dx_\beta.
\end{equation} 
In light of these facts, consider the following definition.
\begin{notation} \label{Nrands}
For $\alpha \in B$ and $j \geq 0$, define
\begin{eqnarray*}
R_{\alpha, j} &:=& \OOp_{\geq 0}\Bigl(x_\alpha^{-j} f'_\alpha(x_\alpha)\Bigr) \dx_\alpha;\\
S_{\alpha, j} &:=& d(yx_\alpha^{-j}) + R_{\alpha, j}.
\end{eqnarray*}
\end{notation}

\begin{remark}\label{rmk:Rjalpha}
Let $a_{\alpha, i}\in k$ be the coefficients of the 
(odd-power) monomials of the polynomials $f_\alpha(x_\alpha)$ defined 
in the partial fraction decomposition \eqref{eqn:pfd}:
\begin{equation*}
f_\alpha(x_\alpha) = \sum_{i=0}^{c_\alpha} a_{\alpha,i} x_\alpha^{2i+1}.
\end{equation*}
Then
\begin{eqnarray*}
R_{\alpha, j} 
&=& \sum_{j/2 \leq i \leq c_\alpha} a_{\alpha, i} x_\alpha^{2i - j} dx_\alpha\\
&=& \sum_{j/2 \leq i \leq c_\alpha}  a_{\alpha, i} \omega_{\alpha, 2i - j + 1}.\\
\end{eqnarray*}
\end{remark}

\begin{lemma} \label{LpoleRS}
Let $\alpha \in B$ and  $j\geq 0$.
\begin{enumerate}

\item The differential form $R_{\alpha, j}$ is regular away from $P_\alpha$, i.e.,
$R_{\alpha,j} \in \Gamma(U_\beta, \Omega^1)$ for all $\beta \in B_\infty - \{\alpha\}$

\item The differential form $S_{\alpha, j}$ is regular at $P_\alpha$ for $0\leq j \leq c_\alpha$,
i.e., $S_{\alpha, j} \in \Gamma(U_\alpha, \Omega^1)$.
\end{enumerate}
\end{lemma}

\begin{proof}
\begin{enumerate}
\item This follows from Remark \ref{rmk:Rjalpha} and 
Equation \eqref{eqn:divomega}.

\item
By Notation \ref{Ntruncate}, Notation \ref{Nrands} and Equations \eqref{dyxj} and \eqref{dyparfrac},
\begin{eqnarray}
S_{\alpha,j}
&=&
d(yx_\alpha^{-j})+\OOp_{\geq 0}(x_\alpha^{-j} f'_\alpha(x_\alpha)) \dx_\alpha\\
&=&
- j x_\alpha^{-j-1} y \dx_\alpha -\OOp_{< 0}(x_\alpha^{-j} f'_\alpha(x_\alpha)) \dx_\alpha
- \sum_{\beta\in B-\set\alpha} x_\alpha^{-j} f'_{\beta} (x_\beta) \dx_\beta
 \label{E3terms}
\end{eqnarray}

For the first part of Equation \eqref{E3terms}, note that the order of vanishing of $x_\alpha^{-j-1} y \dx_\alpha$
at $P_\alpha$ is $2d_\alpha -1 +2j$ by Lemma \ref{Ldivisor} and Equation \eqref{eqn:divdxalpha}, 
and so this term is regular at $P_\alpha$. 

For the second part of Equation \eqref{E3terms}, note that
$\OOp_{< 0}(x_\alpha^{-j} f'_\alpha(x_\alpha))$ is contained in $x_\alpha^{-1}k[x_\alpha^{-1}]$.
Thus $\OOp_{< 0}(x_\alpha^{-j} f'_\alpha(x_\alpha))$ has a zero of order at least $2$ at $P_\alpha$.
As seen in the proof of Lemma \ref{L1formbasis}, $\dx_\alpha$ has a zero of order $d_\alpha-3$ at $P_\alpha$.
Thus $\OOp_{< 0}(x_\alpha^{-j} f'_\alpha(x_\alpha))\dx_\alpha$ is regular at $P_\alpha$.

The last part of Equation \eqref{E3terms} is regular at $P_\alpha$ since $x_\alpha^{-1}$ and $f'_{\beta} (x_\beta) \dx_\beta$ are
regular at $P_\alpha$.
\end{enumerate}
\end{proof}

\subsection{Definition of $\rho$}

We define a $k$-linear morphism $$\rho: \CZ^1(\cU, \cO) \to \CC^0(\cU,\Omega^1)$$ as follows.

\subsubsection{Definition of $\rho$ on $\CB^1(\cU, \cO)$:} \label{sssB}
If $\phi \in \CB^1(\cU, \cO)$,
then $\phi = \delta\kappa$ for some $\kappa\in \CC^0(\cU, \cO)$.
Define
\begin{equation*}
\rho(\phi) := d\kappa,
\end{equation*}
with differentiation performed component-wise. 
This map is well-defined,
since if $\kappa$ is regular at $P\in X(k)$, then so is $d\kappa$.
Moreover, if $\kappa'$ is another element such that $\phi = \delta\kappa'$,
then $\delta(\kappa - \kappa') = 0$ and therefore
$\kappa - \kappa' \in \HH^0(\cU, \cO)$ is constant and annihilated by $d$.
Let $\rho_\beta(\phi)$ denote $(\rho(\phi))_\beta$.

It follows from the definition that $\cC(\rho(\CB^1(\cU, \cO))) = 0$,
since the Cartier operator annihilates all exact differential forms.
Explicitly, the map $\rho$ is computed as follows.
\begin{lemma}\label{lem:rhocobexact}
\begin{enumerate}
\item[(i)] If $\alpha\in B_\infty$ and $h\in \Gamma(U_\alpha, \cO)$, then
$\rho\varphi_\alpha(h|_{U'}) = d\psi_\alpha(h).$
\item[(ii)]
If $\alpha\in B$ and $j\leq 0$, then
$$
\rho\varphi_\alpha((x-\alpha)^{j})
= - \sum_{\gamma\in B_\infty -\set\alpha} d\psi_\gamma((x-\alpha)^{j}).
$$
\end{enumerate}
\end{lemma}
\begin{proof}
Part (i) is immediate from the definition of the map $\rho$ and Lemma \ref{Lcobound}.

Part (ii) follows from part (i), Equation \eqref{eqn:alphaalpha}
and the definition of $\rho$.
\end{proof}

\begin{example}\label{eg:psi} 
The value of $\rho$ on the $1$-coboundary $\varphi_\alpha(f(x) x_\alpha^{-j})$ if $\alpha\in B$ and $j\geq 0$:
Let 
$$
r_{\alpha, j} := \OOp_{>0}\bigl( x_\alpha^{-j} f_\alpha(x_ \alpha) \bigr)
\text{ and }
s_{\alpha, j} := \OOp_{\leq 0}\bigl( x_\alpha^{-j} f_\alpha(x_ \alpha)
\bigr)
+ \sum_{\beta\neq\alpha} x_\alpha^{-j} f_\beta(x_\beta).
$$
Then
$$
f(x) x_\alpha^{-j} = r_{\alpha, j} + s_{\alpha, j},
$$
and $r_{\alpha, j}$ has a pole at $P_\alpha$, but is regular everywhere
else, while $s_{\alpha, j}$ is regular at $P_\alpha$.
Thus,
\begin{equation*}
\varphi_\alpha(f(x) x_\alpha^{-j}) =
\delta \psi_\alpha(s_{\alpha, j}) - 
\sum_{\beta\in B_\infty - \set\alpha} 
\delta\psi_\beta( r_{\alpha,j}).
\end{equation*}

Therefore, for $\beta\neq\alpha$, by Lemma \ref{lem:rhocobexact},
$\rho_\beta \varphi_\alpha(f(x) x_\alpha^{-j}) = -d(r_{\alpha, j})$.
Since $f_\alpha(x_\alpha) \in x_\alpha k[x_\alpha^2]$, this simplifies to
\begin{equation} \label{EhereR}
\rho_\beta \varphi_\alpha(f(x) x_\alpha^{-j}) = 
\begin{cases}
-R_{\alpha, j} & \text{if $j$ is even,}\\
0 & \text{if $j$ is odd.}
\end{cases}
\end{equation}
Similarly,
\begin{equation} \label{EhereS}
\rho_\alpha \varphi_\alpha(f(x) x_\alpha^{-j}) 
= \begin{cases}
-S_{\alpha, j} & \text{if $j$ is even,}\\
 d\bigl(f(x) x_\alpha^{-j}\bigr) & \text{if $j$ is odd.}
\end{cases}
\end{equation}
\end{example}

\subsubsection{Definition of $\rho_\beta$ on $\CZ^1(\cU,\cO)$:} \label{sssZ}

By Lemma \ref{Lnextbasis}, $\CZ^1(\cU, \cO)$ is generated by 
$\CB^1(\cU, \cO)$ and $\phi_{\alpha,j}$ for $\alpha \in B$ and $0 \leq j \leq c_\alpha$.
For $\alpha, \beta \in B$, define
\begin{equation*}
\rho_\beta(\phi_{\alpha, j}) =
\begin{cases}
R_{\alpha,j}
& \text{if $\beta \not = \alpha$},\\
S_{\alpha, j}
& \text{if $\beta=\alpha$},
\end{cases}
\end{equation*}
and extend $\rho_\beta$ to $\CZ^1(\cU,\cO)$ linearly.
For all $\beta \in B -\set\alpha$, note that
\begin{equation*}
\rho_\alpha(\phi_{\alpha, j})=d(yx_\alpha^{-j}) + \rho_\beta(\phi_{\alpha, j}).
\end{equation*}

\begin{lemma}
There is a well-defined map $\rho: \CZ^1(\cU,\cO) \to \CC^0(\cU,\Omega^1)$ given by
$$
\rho:=\bigoplus_{\beta\in B_\infty} \rho_\beta.
$$
\end{lemma}
\begin{proof}
If $\beta\in B_\infty$,
then $\rho_\beta(\CZ^1(\cU,\cO)) \subset \Gamma(U_\beta, \Omega^1)$
by Section \ref{sssB} and Lemma \ref{LpoleRS}.
\end{proof}

Here is an example of a computation of the map $\rho$.

\begin{lemma}\label{lem:twoj}
Let $\alpha \in B$ and $j\geq 0$.  For each $\beta \in B$, in $\Gamma(U_\beta, \Omega^1)$, 
$$
\rho_\beta \varphi_\alpha(y^2 x_\alpha^{-2j})
=
\begin{cases}
0 & \text{if $0 \leq 2j \leq c_\alpha$,}\\
-R_{\alpha, 2j}
& \text{if $2j > c_\alpha$.}
\end{cases}
$$
In particular, 
$\rho \varphi_\alpha(y^2 x_\alpha^{-2j})$
lies in the subspace $W'_{\alpha,\nil}$ of $\HH^0(\cU, \Omega^1)$.

\end{lemma}
\begin{proof}
We have $y^2 x_\alpha^{-2j} = y x_\alpha^{-2j} + f(x) x_\alpha^{-2j}$, and therefore
$\varphi_\alpha(y^2 x_\alpha^{-2j}) = \phi_{\alpha, 2j} + \varphi_\alpha(f(x) x_\alpha^{-2j}).$

Suppose $0 \leq 2j \leq c_\alpha$.  
If $\beta\neq\alpha$, then 
$\rho_\beta(\phi_{\alpha, 2j}) = R_{\alpha,2j} = -\rho_\beta(\varphi_\alpha(f(x) x_\alpha^{-2j}))$ by Equation \eqref{EhereR}. 
By Equation \eqref{EhereS}, 
$\rho_\alpha(\phi_{\alpha, 2j}) = S_{\alpha,2j} = -\rho_\alpha(\varphi_\alpha(f(x) x_\alpha^{-2j}))$.
Thus, $\rho(\phi_{\alpha,2j}) + \rho(\varphi_\alpha(f(x) x_\alpha^{-2j}))=0$.

Now, suppose that $2j>c_\alpha$. Then $y x_\alpha^{-2j}$ is regular
at $P_\alpha$ and therefore $\phi_{\alpha, 2j}$ is a coboundary, with
$\rho(\phi_{\alpha, 2j}) = d \varphi_\alpha(yx_\alpha^{2j}).$
Therefore, for $\beta\neq\alpha$, 
$$
\rho_\beta(\phi_{\alpha, 2j}) + \rho_\beta(\varphi_\alpha(f(x) x_\alpha^{-2j})) = -R_{\alpha, 2j},
$$
and
$$
\rho_\alpha(\phi_{\alpha, 2j}) + \rho_\alpha(\varphi_\alpha(f(x) x_\alpha^{-2j})) 
= d(y x_\alpha^{-2j}) + d(f(x) x_\alpha^{-2j}) - R_{\alpha, 2j} = -R_{\alpha, 2j}.
$$

By Remark \ref{rmk:Rjalpha}, $R_{\alpha,2j} \in \langle \omega_{\alpha, 2i-2j+1} \mid j \leq i \leq c_\alpha \rangle$. 
If $2j>c_\alpha$ and $j \leq i \leq c_\alpha$, then $1 \leq 2i-2j+1 \leq c_\alpha$,
and so $R_{\alpha,2j} \in W'_{\alpha,\nil}$.
Finally, since $\rho_\beta\varphi_\alpha(y^2 x_\alpha^{-2j})$ is independent of
the choice of $\beta\in B_\infty$, 
$\rho\varphi_\alpha(y^2 x_\alpha^{-2j})$
lies in the kernel $\HH^0(\cU, \Omega^1)$ of the coboundary map 
$\delta: C^{0}(\cU, \Omega^1) \to C^{1}(\cU, \Omega^1)$.
\end{proof}

\begin{lemma}\label{lem:deltad}
\begin{enumerate}
\item[(i)]  If $\phi\in \CZ^1(\cU,\cO)$, then $\delta\rho(\phi) = d\phi$.
\item[(ii)] In particular, $\cC(\rho_\alpha(\phi)) = \cC(\rho_\beta(\phi))$ for all $\alpha, \beta\in B_\infty$.
\item[(iii)] For all $\alpha\in B$ and $\beta\in B_\infty$,
$$
\cC(\rho_\beta(\phi_{\alpha, j})) = \cC(R_{\alpha,j}).
$$
\end{enumerate}
\end{lemma}

\begin{proof}
\begin{enumerate}
\item[(i)]
The definition of $\rho_\beta$ implies that $\rho_\alpha(\phi) - \rho_\beta(\phi) = d(\phi)_{\alpha\beta}$
for all $\alpha,\beta\in B_\infty$.
\item[(ii)] This follows from part (i) since the Cartier operator annihilates exact differential forms.
\item[(iii)] This follows from part (ii) and the definition of $\rho_\beta$.
\end{enumerate}
\end{proof}

\begin{remark}
With $a_{\alpha, i}$ defined as in Remark \ref{rmk:Rjalpha}, one can explicitly compute:
$$
\cC(R_{\alpha, j}) = 
\begin{cases}
\sum_{i = (j+1)/2}^{c_\alpha} \sqrt{a_{\alpha, i}}\, \omega_{\alpha, i-{(j-1)}/2}
& \text{if $j$ is odd,}\\
0 & \text{if $j$ is even.}
\end{cases}
$$
In particular, $\cC(R_{\alpha, j}) \in W'_{\alpha,\nil}$.
\end{remark}

\subsection{The $\bE$-module structure of the de Rham cohomology}\label{sub:derham}

Recall, from Section \ref{Sdefderham}, the definition of the non-commutative ring $\bE=k[F,V]$
and the exact sequence:
\begin{equation*}
0 \to \HH^0(X, \Omega^1) \xrightarrow{\lambda} \HdR(X)
\xrightarrow{\gamma} \HH^1(X,\cO) \to 0.
\end{equation*}
Consider the $k$-linear function 
$$
\sigma: \HH^1(X, \cO) \to \HdR(X)
$$
defined by $\sigma(\phi)=(\phi, \rho(\phi))$ for $\phi \in \CZ^1(\cU, \cO)$.

\begin{lemma}\label{Csec}
The function $\sigma$ is a well-defined section of $\gamma: \HdR(X) \to \HH^1(X, \cO)$.
\end{lemma}

\begin{proof}
The function $\sigma$ is clearly a section of $\gamma$.  
It is well-defined because $\sigma(\CB^1(\cU, \cO)) \subset \BdR(\cU)$
by the definition of $\rho_\beta$ on $\CB^1(\cU, \cO)$.
\end{proof}

For $\alpha\in B$, let
$\lambda_{\alpha, j} := \lambda(\omega_{\alpha,j})$ and 
$\sigma_{\alpha, j} := \sigma(\tilde{\phi}_{\alpha,j})$.

\begin{proposition}\label{PactionFV}
For $0 \leq j \leq c_\alpha$, the action of $F$ and $V$ on $\HdR(X)$ is given by:
\begin{enumerate}

\item
$F \lambda_{\alpha, j} = 0$.

\item
$V \lambda_{\alpha, j} = 
\begin{cases}
\lambda_{\alpha, j/2} & \text{if $j$ is even,}\\
0 & \text{if $j$ is odd.}
\end{cases}$

\item
$F\sigma_{\alpha, j} = 
\begin{cases}
\sigma_{\alpha, 2j} & \text{if $j \leq c_\alpha/2$,}\\
\lambda(R_{\alpha, 2j})
& \text{if $j > c_\alpha/2$.}
\end{cases}
$

\item
$V\sigma_{\alpha, j} = 
\begin{cases}
\lambda(\cC(R_{\alpha, j}))
& \text{if $j$ is odd,}\\
0 & \text{if $j$ is even.}
\end{cases}
$
\end{enumerate}
\end{proposition}

\begin{proof} 
\begin{enumerate}
\item This follows from Subsection \ref{sub:fv}.
\item This follows from Lemma \ref{Lcartier1} after applying $\lambda$.
\item
In $\ZdR(\cU)$, 
\begin{eqnarray*}
F(\sigma_{\alpha,j}) 
&=& \bigl(F\phi_{\alpha, j}, 0 \bigr)\\
&=& \bigl(\varphi_\alpha(y^2 x_\alpha^{-2j}), \rho\varphi_\alpha(y^2 x_\alpha^{-2j}) \bigr)
- \bigl(0, \rho\varphi_\alpha(y^2 x_\alpha^{-2j}) \bigr) \\
&=& \sigma\varphi_\alpha(y^2 x_\alpha^{-2j})
- \bigl(0, \rho\varphi_\alpha(y^2 x_\alpha^{-2j}) \bigr).
\end{eqnarray*}
Since $y^2 x_\alpha^{-2j} = y x_\alpha^{-2j} + f(x) x_\alpha^{-2j}$,
linearity of $\sigma$ and $\varphi_\alpha$ yields that
$$
\sigma\varphi_\alpha(y^2 x_\alpha^{-2j}) = \sigma\varphi_\alpha(y x_\alpha^{-2j})  
+ \sigma\varphi_\alpha\bigl(f(x) x_\alpha^{-2j}\bigr).
$$
The term $\sigma\varphi_\alpha(f(x) x_\alpha^{-2j})$ is a coboundary by Lemma \ref{lem:boundary}(i).
The term $\sigma\varphi_\alpha(y x_\alpha^{-2j})$ equals 
$\sigma_{\alpha, 2j}$ if $0\leq 2j \leq c_\alpha$,
and is a coboundary if $2j>c_\alpha$ by Lemma \ref{lem:boundary}(ii).
By Lemma \ref{lem:twoj},
$$
\bigl(0, \rho\varphi_\alpha(y^2 x_\alpha^{-2j}) \bigr) =
\begin{cases}
0 & \text{if $0 \leq 2j \leq c_\alpha$,}\\
-\lambda(R_{\alpha,2j})
& \text{if $2j > c_\alpha$}.
\end{cases}
$$
\item
Since $V(\phi, \rho(\phi)) = (0, \cC(\rho(\phi)))$, the desired result follows by Lemma \ref{lem:deltad}(iii).
\end{enumerate}
\end{proof}

Consider the subspaces of $\HdR(X)$ given by:
\begin{eqnarray*}
W_{\alpha, \sss} &:=& \gen{\lambda_{\alpha, 0}-\lambda_{0,0}, \sigma_{\alpha,0}},\\
W_{\alpha, \nil} &:=&\gen{\lambda_{\alpha, j}, \sigma_{\alpha,j} \mid 1\leq j \leq c_\alpha}.
\end{eqnarray*}

\begin{theorem} \label{Thdr}
The subspaces $W_{\alpha,\sss}$ and $W_{\alpha,\nil}$ of $\HdR(X)$ are stable
under the action of Frobenius and Verschiebung for each $\alpha \in B$.
There is an isomorphism of $\bE$-modules:
$$
\HdR(X) = 
\bigoplus_{\alpha\in B'} W_{\alpha,\sss} 
\oplus
\bigoplus_{\alpha\in B} W_{\alpha,\nil}.
$$
\end{theorem}

\begin{proof}
The stability is immediate by Proposition \ref{PactionFV},
Remark \ref{rmk:Rjalpha}, and Lemma \ref{lem:twoj}.
The decomposition follows from Corollary \ref{Csec} and 
Lemmas \ref{L1form} and \ref{H1decomposition}.
\end{proof}

Theorem \ref{Tmainresult} is immediate from Theorem \ref{Thdr}.

\section{Results on the Ekedahl-Oort type} \label{Sreseotype}

For a natural number $c$, 
let $G_c$ be the unique symmetric ${\rm BT}_1$ group scheme of rank $p^{2c}$ with Ekedahl-Oort type
$[0,1,1,2,2, \ldots,  \lfloor c/2 \rfloor]$.
In other words, this means that there is a final filtration $N_1 \subset N_2 \subset \cdots \subset N_{2c}$ 
of $D(G_c)$ as a $k$-vector space, 
which is stable under the action of $V$ and $F^{-1}$ and with $i={\rm dim}(N_i)$, 
such that ${\rm dim}(V(N_i))=\lfloor i/2 \rfloor$.
In Section \ref{Sfinalfil}, we prove that group schemes of the form $G_c$ appear in the decomposition of 
$J_X[2]$ when $X$ is a hyperelliptic $k$-curve.  
In Section \ref{SdesEO}, we describe the Dieudonn\'e module of $G_c$ for arbitrary $c$ and give examples.  

\subsection{The final filtration for hyperelliptic curves in characteristic $2$} \label{Sfinalfil}

Suppose $X$ is a hyperelliptic $k$-curve with affine equation $y^2-y=f(x)$ as described in Notation \ref{Nsetup}.
For $\alpha \in B$, recall that $c_\alpha=(d_\alpha-1)/2$, 
where $d_\alpha$ is the ramification invariant of $X$ above $\alpha$.
Recall the subspaces $W_{\alpha, \nil}$ of $\HdR(X)$ from Section \ref{sub:derham}.
Define subspaces $N_{\alpha, i}$ of $W_{\alpha,\nil}$ for $0\leq i \leq 2c_\alpha$
as follows: $N_{\alpha, 0} := \set{0}$ and
$$
N_{\alpha, i} := \begin{cases}
\gen{\lambda_{\alpha, j} \mid 1\leq j \leq i} & \text{if $1\leq i\leq c_\alpha$},\\
N_{\alpha, c_\alpha} \oplus \gen{\sigma_{\alpha, j} \mid 1\leq j \leq i}
& \text{if $c_\alpha+1\leq i\leq 2c_\alpha$}.
\end{cases}
$$

\begin{proposition} \label{Pfiltration} 
The filtration $N_{\alpha, 0} \subset N_{\alpha, 1} \subset N_{\alpha, 2} \subset \cdots 
\subset N_{\alpha, 2c_\alpha}$ is a final filtration of $W_{\alpha, \nil}$ for each $\alpha \in B$.
Furthermore, $V(N_{\alpha, i})=N_{\alpha, \lfloor i/2 \rfloor}$.
\end{proposition}

\begin{proof}
Let $0 \leq i \leq 2c_\alpha$.  One sees that $\dim(N_{\alpha, i}) = i$.  By Proposition \ref{PactionFV}, 
$V(N_{\alpha, i}) = N_{\alpha, \lfloor i/2 \rfloor}$ and $F^{-1}(N_{\alpha, i}) = N_{\alpha, c_\alpha+\lceil i/2 \rceil}$.
Thus the filtration 
$N_{\alpha, 0} \subset N_{\alpha, 1} \subset N_{\alpha, 2} \subset \cdots \subset N_{\alpha, 2c_\alpha}$ 
is stable under the action of $V$ and $F^{-1}$.
\end{proof}

\begin{theorem} \label{Tintrorepeat}
Suppose $X$ is a hyperelliptic $k$-curve with affine equation $y^2-y=f(x)$ as described in Notation \ref{Nsetup}.
Then the $2$-torsion group scheme of $X$ decomposes as
$$
J_X[2] \simeq  (\ZZ/2 \oplus \mu_2)^r \oplus \bigoplus_{\alpha \in B} G_{c_{\alpha}},
$$
and the $a$-number of $X$ is $a_X=(g+1-\#\{\alpha \in B \  \mid \ d_\alpha \equiv 1 \bmod 4\})/2$.
\end{theorem}

\begin{proof}
By \cite[Section 5]{Oda}, there is an isomorphism of $\bE$-modules between the Dieudonn\'e module 
$D(J_X[2])$ and the de Rham cohomology $\HdR(X)$.  
By Theorem \ref{Thdr}, 
there is an isomorphism of $\bE$-modules:
$$
\HdR(X) = 
\bigoplus_{\alpha\in B'} W_{\alpha,\sss} 
\oplus
\bigoplus_{\alpha\in B} W_{\alpha,\nil}.
$$
If $\alpha \in B'$, then $W_{\alpha, \sss}$ is isomorphic to 
$\bE/\bE(F, 1-V) \oplus \bE/\bE(V, 1-F) \simeq D(\ZZ/2 \oplus \mu_2)$.
Finally, Proposition \ref{Pfiltration} shows that $W_{\alpha,\nil} \simeq D(G_{c_\alpha})$, which 
completes the proof of the statement about $J_X[2]$.
The statement about $a_X$ can be found in Proposition \ref{Panumber}.
\end{proof}

As a corollary, we highlight the special case when $r=0$ (i.e., $f(x) \in k[x]$).
Corollary \ref{Cp=2} is stated without proof in \cite[3.2]{V:cycles}.

\begin{corollary} \label{Cp=2}
Let $X$ be a hyperelliptic $k$-curve of genus $g$ and $p$-rank $r=0$.  
Then the Ekedahl-Oort type of $J_X[2]$ is 
$[0,1,1,2,2,\ldots, \lfloor g/2 \rfloor]$ and the $a$-number is $a_X=\lfloor (g+1)/2 \rfloor$. 
\end{corollary}

\begin{proof}
This is a special case of Theorem \ref{Tintrorepeat} where $\#B=1$.
\end{proof}

The next result is included to emphasize that Theorem \ref{Tintrorepeat} gives a complete classification of 
the $2$-torsion group schemes which occur as $J_X[2]$ when $X$ is a hyperelliptic $k$-curve.

\begin{corollary}
Let $G$ be a symmetric $BT_1$ group scheme of rank $p^{2g}$.  Let $0 \leq r \leq g$.
Then $G \simeq J_X[2]$ for some hyperelliptic $k$-curve $X$ of genus $g$ and $p$-rank $r$
if and only if there exist $c_1, \ldots, c_{r+1} \in {\mathbb N} \cup \{0\}$ such that $\sum_{i=1}^{r+1}c_i=g-r$ and 
such that 
$$G \simeq  (\ZZ/2 \oplus \mu_2)^r \oplus \bigoplus_{\alpha \in B} G_{c_{\alpha}}.$$
\end{corollary}

\begin{proof}
This is immediate from Theorem \ref{Tintrorepeat}.
\end{proof}

\begin{remark}
For fixed $g$, the number of isomorphism classes of symmetric $BT_1$ group schemes of rank $p^{2g}$
that occur as $J_X[2]$ for some hyperelliptic $k$-curve $X$ of genus $g$ equals the number of partitions of $g+1$.
To see this, note that the isomorphism class of $J_X[2]$ is determined by the multi-set
$\{d_1, \ldots, d_{r+1}\}$ where $d_i=2c_i+1$ and $\sum_{i=1}^{r+1}(d_i+1)=2g+2$.
So the number of isomorphism classes equals the number of partitions of $2g+2$ into even integers.  
\end{remark}

\begin{remark}
The examples in Section \ref{SdesEO} show that the factors appearing in the decomposition of $J_X[2]$ 
in Theorem \ref{Tintrorepeat} may not be indecomposable as symmetric ${\rm BT}_1$ group schemes.
\end{remark}

\subsection{Description of a particular Ekedahl-Oort type} \label{SdesEO}

Recall that $G_c$ is the unique symmetric ${\rm BT}_1$ group scheme of rank $p^{2c}$ with Ekedahl-Oort type
$[0,1,1,2,2, \ldots,  \lfloor c/2 \rfloor]$.  Recall that $\bE=k[F,V]$ is the non-commutative ring defined in Section \ref{Sdefderham}.
In this section, we describe the Dieudonn\'e module $D(G_c)$.  
We start with some examples to motivate the notation.
The examples show that $G_c$ is sometimes indecomposable and sometimes decomposes into
symmetric ${\rm BT}_1$ group schemes of smaller rank.
The first four examples were found using pre-existing tables.

\begin{example}
\begin{enumerate}
\item For $c=1$, the Ekedahl-Oort type is $[0]$. 
This Ekedahl-Oort type occurs for the $p$-torsion group scheme of a supersingular elliptic curve.
See  \cite[Ex.\ A.3.14]{G:book} or \cite[Ex.\ 2.3]{Pr:groupscheme} for a description of $G_1$.
It has Dieudonn\'e module $\bE/\bE(F + V)$. 

\item For $c=2$, the Ekedahl-Oort type is $[0,1]$.  
This Ekedahl-Oort type occurs for the $p$-torsion group scheme of a supersingular abelian surface 
which is not superspecial. 
See \cite[Ex.\ A.3.15]{G:book} or \cite[Ex.\ 2.3]{Pr:groupscheme} for a description of $G_2$.
It has Dieudonn\'e module $\bE/\bE(F^2 + V^2)$.

\item For $c=3$, the Ekedahl-Oort type is $[0,1,1]$.
This Ekedahl-Oort type occurs for an abelian threefold with $p$-rank $0$ and $a$-number $2$ whose $p$-torsion 
is indecomposable as a symmetric $BT_1$ group scheme.
By \cite[Lemma 3.4]{Pr:groupscheme}, $G_3$ has Dieudonn\'e module
$$\bE/\bE(F^{2}+V) \oplus \bE/\bE(V^{2}+F).$$

\item For $c=4$, the Ekedahl-Oort type is $[0,1,1,2]$.
This Ekedahl-Oort type occurs for an abelian fourfold with $p$-rank $0$ and $a$-number $2$ 
whose $p$-torsion decomposes as a direct sum of symmetric $BT_1$ group schemes of rank $p^2$ and $p^6$.
By \cite[Table 4.4]{Pr:groupscheme}, $G_4$ has Dieudonn\'e module
$$\bE/\bE(F + V) \oplus \bE/\bE(F^{3}+V^{3}).$$
\end{enumerate}
\end{example}

We now provide an algorithm to determine the Dieudonn\'e module $D(G_c)$ for all $c \in \NN$ following
the method of \cite[Section 9.1]{O:strat}.

\begin{proposition} \label{Pdieu}
The Dieudonn\'e module $D(G_c)$ is the $\bE$-module generated 
as a $k$-vector space by $\{X_1, \ldots, X_c, Y_1, \ldots, Y_c\}$ 
with the actions of $F$ and $V$ given by:
\begin{enumerate}
\item
$F(Y_j) = 0$.

\item
$V(Y_j) = 
\begin{cases}
Y_{2j} & \text{if $j \leq c/2$,}\\
0 & \text{if $j > c/2$.}
\end{cases}$

\item
$F(X_i)=
\begin{cases}
X_{j/2} & \text{if $j$ is even,}\\
Y_{c-(j-1)/2} & \text{if $j$ is odd.}
\end{cases}$

\item 
$V(X_j)=
\begin{cases}
0 & \text{if $j \leq (c-1)/2$,}\\
-Y_{2c-2j+1} & \text{if $j > (c-1)/2$.}
\end{cases}$
\end{enumerate}
\end{proposition}

\begin{proof}
By definition of $G_c$, there is a final filtration $N_1 \subset N_2 \subset \cdots \subset N_{2c}$ 
of $D(G_c)$ as a $k$-vector space, 
which is stable under the action of $V$ and $F^{-1}$ and with $i={\rm dim}(N_i)$, 
such that $\nu_i:={\rm dim}(V(N_i))=\lfloor i/2 \rfloor$.
This implies that $\nu_i=\nu_{i-1}$ if and only if $i$ is odd.
In the notation of \cite[Section 9.1]{O:strat}, this yields
$m_i=2i$ and $n_i=2g-2i+1$ for $1 \leq i \leq g$;
also, let
$$Z_i=
\begin{cases}
X_{i/2} & \text{if $i$ is even,}\\
Y_{c-(i-1)/2} & \text{if $i$ is odd}.
\end{cases}$$
By \cite[Section 9.1]{O:strat}, for $1 \leq i \leq g$, the action of $F$ is given by
$F(Y_i)=0$ and $F(X_i)=Z_i$; and the action of $V$ is given by
$V(Z_i)=0$ and $V(Z_{2g-i+1})= (-1)^{i-1}Y_i$.
\end{proof}

More notation is needed to give an explicit description of $D(G_c)$.

\begin{notation} \label{Nbij}
Let $c \in \NN$ be fixed. 
Let $I=\{j \in \NN \mid \lceil (c+1)/2 \rceil \leq j \leq c\}$ which is a set of cardinality $\lfloor (c+1)/2 \rfloor$.
For $j \in I$, let $\ell(j)$ be the odd part of $j$ and let $e(j)$ be the non-negative integer such that $j=2^{e(j)}\ell(j)$.
Let $s(j)=c-(\ell(j)-1)/2$.  
One can check that $\{s(j) \mid j \in I\}=I$.
Also, let $m(j)=2c-2j+1$ and let $\epsilon(j)$ be the non-negative integer 
such that $t(j):=2^{\epsilon(j)}m(j) \in I$.  
One can check that $\{t(j) \mid j \in I\}=I$.
Thus there is a unique bijection $\iota:I \to I$ such that $t(\iota(j))=s(j)$ for each $j \in I$.
\end{notation}

\begin{proposition} \label{Palg}
Recall Notation \ref{Nbij}.
For $c \in \NN$, the set $\{X_j \mid j \in I\}$ generates the Dieudonn\'e module $D(G_c)$ as an $\bE$-module
subject to the relations: $F^{e(j)+1}(X_j) + V^{\epsilon(\iota(j))+1}(X_{\iota(j)})$ for $j \in I$.
Also, $\{X_j \mid j \in I\}$ is a basis for $D(G_c)$ modulo $D(G_c)(F,V)$.
\end{proposition}

\begin{proof}
By Proposition \ref{Pdieu}, $F^{e(j)}(X_j)=X_{\ell(j)}$ and $F(X_{\ell(j)})=Y_{s(j)}$.
Also, $V(X_j)=-Y_{m(j)}$ and so $V^{\epsilon(j)+1}(X_j)=-Y_{t(j)}$.
This yields the stated relations.
To complete the first claim, it suffices to show that the span of $\{X_j \mid j \in I\}$ under the action of $F$ and $V$ 
contains the $k$-module generators of $D(C_c)$ listed in Proposition \ref{Pdieu}.
This follows from the observations that $X_i=F(X_{2i})$ if $1 \leq i \leq \lfloor c/2 \rfloor$,
that $Y_i=V(Y_{i/2})$ if $i$ is even and $Y_i=V(-X_{c-(i-1)/2})$ if $i$ is odd.   
By \cite[5.2.8]{LO}, the dimension of $D(G_c)$ modulo $D(G_c)(F,V)$ equals the $a$-number. 
Since $a=|I|$ by Corollary \ref{Cp=2}, it follows that the set $|I|$ of generators of $D(G_c)$ is 
linearly independent modulo $D(G_c)(F,V)$.
\end{proof}

Here are some more examples.  The columns of the following table list: the value of $c$; 
the generators of $D(G_c)$ as an $\bE$-module (where $X_{i_1} - X_{i_2}$ denotes $\{X_i \mid i_1 \leq i \leq i_2\}$); 
the relations among these generators; and the number of (possibly unsymmetric)
indecomposable summands of $D(G_c)$.  The table can be verified in two ways: first, by checking it with 
Proposition \ref{Palg}; second, by computing the action of $F$ and $V$ on a $k$-basis for $D(G_c)$, using this to
construct a final filtration of $D(G_c)$ stable under $V$ and $F^{-1}$, and then checking that it matches the 
Ekedahl-Oort type of $G_c$.  In Example \ref{Ec=7}, we illustrate the second method. 

\begin{center}
$\begin{array}{|c|c|l|c|}
\hline
c& \text{generators} & \text{relations} & \text{\# summands}\\
\hline
5 & X_3 - X_5 & FX_3+V^3X_5, F^3X_4 + VX_3, FX_5+ VX_4 & 1 \\
\hline
6 & X_4 - X_6 & F^3X_4 +V^2X_5, FX_5 +V^3X_6, F^2X_6+VX_4 & 1\\ 
\hline
7 & X_4 - X_7 & F^3X_4+VX_4, FX_5+VX_5, F^2X_6 + V^2X_6, FX_7+V^3X_7 & 4\\
\hline
8 & X_5 - X_8 & FX_5+V^2X_7, F^2X_6+VX_5, FX_7 + VX_6, F^4X_8+V^4X_8 & 2\\
\hline 
9 & X_5 - X_9 & FX_5+VX_6, F^2X_6+V^4X_9, FX_7+V^2X_8, & \\
& & F^4X_8 +VX_5, FX_9+VX_7 & 1\\
\hline
10 & X_6 - X_{10} & F^2X_6+VX_6, FX_7+VX_7, F^4X_8+V^2X_8, & \\ 
& & FX_9+V^2X_9, F^2X_{10}+V^4X_{10} & 5\\
\hline
\end{array}$
\end{center}

\begin{example} \label{Ec=7}
For $c=7$, the group scheme $G_7$ with Ekedahl-Oort type $[0,1,1,2,2,3,3]$
is isomorphic to a direct sum of symmetric $BT_1$ group schemes of ranks
$p^2$, $p^4$ and $p^{8}$
and has Dieudonn\'e module
$${\bf M}:=\bE/\bE(F + V) \oplus \bE/\bE(F^2+V^{2})\oplus \bE/\bE(V+F^{3})
\oplus \bE/\bE(F^{3}+V).$$
\end{example}

\begin{proof}
Let $\set{1_A, V_A}$ be the basis of the submodule $A=\bE/\bE(F + V)$ of ${\bf M}$; 
let $\set{1_B, V_B, V_B^2, F_B^2}$ be the basis of the submodule $B=\bE/\bE(F^2 + V^2)$;
let $\set{1_C, V_C, V_C^2, V_C^3}$ be the basis of the submodule $C=\bE/\bE(F +V^3)$;
and let $\set{1_{C'}, F_{C'}, F_{C'}^2, F_{C'}^3}$ be the basis of the submodule $C'=\bE/\bE(F^3 + V)$.
The action of Frobenius and Verschiebung on the elements of these bases is:
\begin{center}
\begin{tabular}{|c||cc|cccc|cccc|cccc|}
\hline
$x$ & $1_A$ & $V_A$ & $1_B$ & $V_B$ & $V_B^2$ & $F_B$ & $1_C$ & $V_C$ &
$V_C^2$ & $V_C^3$ & $1_{C'}$ & $F_{C'}$ & $F_{C'}^2$ & $F_{C'}^3$ \\
\hline 
\hline 
$Vx$ & $V_A$ & $0$ & $V_B$ & $V_B^2$ & $0$ & $0$ & $V_C$ & $V_C^2$ & $V_C^3$
& $0$ & $F_{C'}^3$ & $0$ & $0$ & $0$ \\
\hline
$Fx$ & $V_A$ & $0$ & $F_B$ & $0$ & $0$ & $V_B^2$ & $V_C^3$ & $0$ & $0$ & $0$ &
$F_{C'}$ & $F_{C'}^2$ & $F_{C'}^3$ & $0$ \\
\hline
\end{tabular}
\end{center}
To verify the proposition, one can repeatedly apply $V$ and $F^{-1}$ to construct a filtration  
$N_1 \subset N_2 \subset \cdots \subset N_{14}$ of ${\bf M}$ as a $k$-vector space 
which is stable under the action of $V$ and $F^{-1}$ such that $i={\rm dim}(N_i)$. 
To save space, we summarize the calculation by listing a generator $t_i$ for $N_i/N_{i-1}$:
\begin{center}
\begin{tabular}{|c|cccccccccccccc|}
\hline
$i$ & 1 & 2 & 3 & 4 & 5 & 6 & 7 & 8 & 9 & 10 & 11 & 12 & 13 & 14\\
\hline 
$t_i$ & $V_C^3$&$V_C^2$&$V_B^2$&$V_C$&$V_{A}$ &$ F_{C'}^3$&$V_B$ &
$1_C$&$F_{C'}^2$&$1_{A}$&$F_B$&$F_{C'}$&$1_{C'}$&$1_{B}$ \\
\hline
\end{tabular}
\end{center}

Then one can check that $V(N_i)=N_{\lfloor i/2 \rfloor}$ and $F^{-1}(N_i)=N_{7+\lceil i/2 \rceil}$, 
which verifies that the Ekedahl-Oort type of ${\bf M}$ is $[0,1,1,2,2,3,3]$.
\end{proof}

\begin{remark}
One could ask when $D(G_c)$ decomposes as much as numerically possible, 
in other words, when the number of (possibly unsymmetric) indecomposable summands of $D(G_c)$ equals the $a$-number.
For example, $D(G_c)$ has this property when $c \in \{1-4, 7, 10\}$ but not when $c \in \{5,6,8,9\}$.
This phenomenon occurs if and only if the bijection $\iota$ from Notation \ref{Nbij} is the identity.
\end{remark}

\begin{remark}
The group scheme $G_8$ decomposes as the direct sum of two indecomposable symmetric ${\rm BT}_1$ group schemes, 
one whose Ekedahl-Oort type is $[0,0,1,1]$, and the other whose covariant Dieudonn\'e module is $\bE/\bE(F^4+V^4)$. 
We take this opportunity to note that there is a mistake in \cite[Example in Section 3.3]{Pr:groupscheme}.
The covariant Dieudonn\'e module of $I_{4,3}=[0,0,1,1]$ is stated incorrectly.
To fix it, consider the method of \cite[Section 9.1]{O:strat}.
Consider the $k$-vector space of dimension $8$ generated by $X_1, \ldots, X_4$ and $Y_1, \ldots Y_4$.
Consider the operation $F$ defined by: $F(Y_i)=0$ for $1 \leq i \leq 4$ and 
$$F(X_1)=Y_4; \ F(X_2)=Y_3; \ F(X_3)=X_1; \ F(X_4)=Y_2.$$
Consider the operation $V$ defined by: 
$$V(X_1)=0; \ V(X_2)=-Y_4; \ V(X_3)=-Y_2; \ V(X_4)=-Y_1;$$
and 
$$V(Y_1)=Y_3; V(Y_2)=0; \ V(Y_3)=0; V(Y_4)=0.$$  
Thus $D(I_{4,3})$ is generated by $X_2,X_3, X_4$ modulo the relations 
$$FX_2+V^2X_4, F^2X_3 +VX_2, VX_3 +FX_4.$$ 
\end{remark}

\subsection{Newton polygons}

There are several results in characteristic $2$ about the Newton polygons of hyperelliptic 
(e.g., Artin-Schreier) curves $X$ of genus $g$ and $2$-rank $0$.
For example, \cite[Remark 3.2]{blache08} states that if $2^{n-1} - 1 \leq g \leq 2^n-2$, 
then the generic first slope of the Newton polygon of an Artin-Schreier curve of genus $g$ and $2$-rank $0$ is $1/n$.
This statement is made more precise in \cite[Thm.\ 4.3]{blache09}.  
See also earlier work in \cite[Thm.\ 1.1(III)]{SZ:02}.  


The Ekedahl-Oort type of $J_X[2]$ gives information about the Newton polygon of $X$, 
but does not determine it completely.  
Using Corollary \ref{Cp=2} and \cite[Section 3.1, Theorem 4.1]{harashita}, 
one can show that the first slope of the Newton polygon of $X$ is at least $1/n$.  
Since this is weaker than \cite[Thm.\ 4.3]{blache09}, we do not include the details.

More generally, one could consider the case that $X$ is a hyperelliptic $k$-curve of genus $g$ and 
arbitrary $p$-rank.  One could use Theorem \ref{Tintrorepeat} 
to give partial information (namely a lower bound) for the Newton polygon of $X$.



\bibliographystyle{plain}
\bibliography{EOhypstrata}

\begin{thebibliography}{10}

\bibitem{blache09}
R.~Blache.
\newblock First vertices for generic {N}ewton polygons, and {$p$}-cyclic
  coverings of the projective line.
\newblock arXiv:0912.2051.

\bibitem{blache08}
R.~Blache.
\newblock {$p$}-density, exponential sums and {A}rtin-{S}chreier curves.
\newblock arXiv:0812.3382.

\bibitem{MAGMA}
W.~Bosma, J.~Cannon, and C.~Playoust.
\newblock The magma algebra system i: the user language.
\newblock {\em J. Symb. Comput.}, 24(3-4):235--265, 1997.

\bibitem{B}
I.~Bouw.
\newblock The {$p$}-rank of ramified covers of curves.
\newblock {\em Compositio Math.}, 126(3):295--322, 2001.

\bibitem{Cartier}
P~Cartier.
\newblock Une nouvelle op{\'e}ration sur les formes diff{\'e}rentielles.
\newblock {\em Lecture Notes Ser. Comput.}, 244:426--428, 1957.

\bibitem{Crew}
R.~Crew.
\newblock \'{E}tale {$p$}-covers in characteristic {$p$}.
\newblock {\em Compositio Math.}, 52(1):31--45, 1984.

\bibitem{Demazure}
M.~Demazure.
\newblock {\em Lectures on {$p$}-divisible groups}, volume 302 of {\em Lecture
  Notes in Mathematics}.
\newblock Springer-Verlag, Berlin, 1986.
\newblock Reprint of the 1972 original.

\bibitem{Ekedahl}
T.~Ekedahl.
\newblock On supersingular curves and abelian varieties.
\newblock {\em Math. Scand.}, 60(2):151--178, 1987.

\bibitem{El:bound}
A.~Elkin.
\newblock The rank of the {C}artier operator on cyclic covers of the projective
  line.
\newblock to appear in J. Algebra, arXiv:0708.0431.

\bibitem{G:book}
E.~Goren.
\newblock {\em Lectures on {H}ilbert modular varieties and modular forms},
  volume~14 of {\em CRM Monograph Series}.
\newblock American Mathematical Society, Providence, RI, 2002.
\newblock With M.-H. Nicole.

\bibitem{harashita}
S.~Harashita.
\newblock Ekedahl-{O}ort strata and the first {N}ewton slope strata.
\newblock {\em J. Algebraic Geom.}, 16(1):171--199, 2007.

\bibitem{John}
O.~Johnston.
\newblock A note on the $a$-numbers and $p$-ranks of {K}ummer covers.
\newblock arXiv:0710.2120.

\bibitem{Kraft}
H.~Kraft.
\newblock Kommutative algebraische $p$-gruppen (mit anwendungen auf
  $p$-divisible gruppen und abelsche variet\"aten).
\newblock manuscript, University of Bonn, September 1975, 86 pp.

\bibitem{LO}
K.-Z. Li and F.~Oort.
\newblock {\em Moduli of supersingular abelian varieties}, volume 1680 of {\em
  Lecture Notes in Mathematics}.
\newblock Springer-Verlag, Berlin, 1998.

\bibitem{madden}
D.~Madden.
\newblock Arithmetic in generalized {A}rtin-{S}chreier extensions of {$k(x)$}.
\newblock {\em J. Number Theory}, 10(3):303--323, 1978.

\bibitem{M:group}
B.~Moonen.
\newblock Group schemes with additional structures and {W}eyl group cosets.
\newblock In {\em Moduli of abelian varieties (Texel Island, 1999)}, volume 195
  of {\em Progr. Math.}, pages 255--298. Birkh\"auser, Basel, 2001.

\bibitem{NS:hyper}
E.~Nart and D.~Sadornil.
\newblock Hyperelliptic curves of genus three over finite fields of even
  characteristic.
\newblock {\em Finite Fields Appl.}, 10(2):198--220, 2004.

\bibitem{Oda}
T.~Oda.
\newblock The first de {R}ham cohomology group and {D}ieudonn{\'e} modules.
\newblock {\em Ann. Sci. Ecole Norm. Sup. (4)}, 2:63--135, 1969.

\bibitem{O:strat}
F.~Oort.
\newblock A stratification of a moduli space of abelian varieties.
\newblock In {\em Moduli of abelian varieties (Texel Island, 1999)}, volume 195
  of {\em Progr. Math.}, pages 345--416. Birkh\"auser, Basel, 2001.

\bibitem{Pr:groupscheme}
R.~Pries.
\newblock A short guide to {$p$}-torsion of abelian varieties in characteristic
  {$p$}.
\newblock In {\em Computational arithmetic geometry}, volume 463 of {\em
  Contemp. Math.}, pages 121--129. Amer. Math. Soc., Providence, RI, 2008.
\newblock math.NT/0609658.

\bibitem{SZ:02}
J.~Scholten and H.~J. Zhu.
\newblock Hyperelliptic curves in characteristic 2.
\newblock {\em Int. Math. Res. Not.}, (17):905--917, 2002.

\bibitem{Se:lf}
J.-P. Serre.
\newblock {\em Corps Locaux}.
\newblock Hermann, 1968.

\bibitem{sti}
H.~Stichtenoth.
\newblock {\em Algebraic function fields and codes}, volume 254 of {\em
  Graduate Texts in Mathematics}.
\newblock Springer-Verlag, Berlin, second edition, 2009.

\bibitem{Sullivan}
F.~Sullivan.
\newblock {$p$}-torsion in the class group of curves with too many
  automorphisms.
\newblock {\em Arch. Math. (Basel)}, 26:253--261, 1975.

\bibitem{V:cycles}
G.~van~der Geer.
\newblock Cycles on the moduli space of abelian varieties.
\newblock In {\em Moduli of curves and abelian varieties}, Aspects Math., E33,
  pages 65--89. Vieweg, Braunschweig, 1999.
\newblock arXiv:alg-geom/9605011.

\bibitem{Y}
N.~Yui.
\newblock On the {J}acobian varieties of hyperelliptic curves over fields of
  characteristic {$p>2$}.
\newblock {\em J. Algebra}, 52(2):378--410, 1978.

\end{thebibliography}

\end{document}